\documentclass[a4paper,11pt]{amsart}
\usepackage[english]{babel}
\usepackage{amsmath}
\usepackage{amssymb}
\usepackage{mathrsfs}
\usepackage{enumerate}
\usepackage{mathtools}
\usepackage{tikz,tkz-euclide}
\usepackage{pgfplots}
\usepackage{hyperref}
\usetikzlibrary{arrows}
\usepackage{makecell}

\usepackage{tcolorbox}
\usepackage[]{algorithm2e}

\newtheorem{thm}{Theorem}[section]
\newtheorem{Lemma}[thm]{Lemma}
\newtheorem{Proposition}[thm]{Proposition}
\newtheorem{Corollary}[thm]{Corollary}

\newtheorem*{thm*}{Theorem}

\theoremstyle{definition}

\newtheorem{Definition}[thm]{Definition}
\newtheorem{Remark}[thm]{Remark}

\newtheorem{Example}[thm]{Example}

\definecolor{wwwwww}{rgb}{0.4,0.4,0.4}

\DeclareMathOperator{\Pic}{Pic}

\DeclareMathOperator{\expdim}{expdim}

\DeclareMathOperator{\Sym}{Sym}

\DeclareMathOperator{\Sec}{\mathbb{S}ec}

\hypersetup{pdfpagemode=UseNone}
\hypersetup{pdfstartview=FitH}
		
\begin{document}

\title{On secant defectiveness and identifiability of Segre-Veronese varieties}

\author[Antonio Laface]{Antonio Laface}
\address{\sc Antonio Laface\\
Departamento de Matematica, Universidad de Concepci\'on\\
Casilla 160-C, Concepci\'on\\
Chile}
\email{alaface@udec.cl}

\author[Alex Massarenti]{Alex Massarenti}
\address{\sc Alex Massarenti\\ Dipartimento di Matematica e Informatica, Universit\`a di Ferrara, Via Machiavelli 30, 44121 Ferrara, Italy}
\email{alex.massarenti@unife.it}

\author[Rick Rischter]{Rick Rischter}
\address{\sc Rick Rischter\\
Universidade Federal de Itajub\'a (UNIFEI)\\ 
Av. BPS 1303, Bairro Pinheirinho\\ 
37500-903, Itajub\'a, Minas Gerais\\ 
Brazil}
\email{rischter@unifei.edu.br}

\date{\today}

\begin{abstract}
We give an almost asymptotically sharp bound for the non secant defectiveness and identifiability of Segre-Veronese varieties. We also provide new examples of defective Segre-Veronese varieties.
\end{abstract}

\maketitle
\setcounter{tocdepth}{1}
\tableofcontents

\section{Introduction}
The \textit{$h$-secant variety} $\mathbb{S}ec_{h}(X)$ of a non-degenerate $n$-dimensional variety $X\subseteq\mathbb{P}^N$ is the Zariski closure of the union of all linear spaces spanned by collections of $h$ points of $X$. The \textit{expected dimension} of $\mathbb{S}ec_{h}(X)$ is $\expdim(\mathbb{S}ec_{h}(X)):= \min\{nh+h-1,N\}$. In general, the actual dimension of $\mathbb{S}ec_{h}(X)$ may be smaller than the expected one. In this case, following \cite[Section 2]{CC10} we say that $X$ is \textit{$h$-defective}. The problem of determining the actual dimension of secant varieties, and its relation with the dimension of certain linear systems of hypersurfaces with double points, have a very long history in algebraic geometry, and can be traced back to the Italian school \cite{Sc08}, \cite{Se01}, \cite{Te11}. Since then the geometry of secant varieties has been studied and used by many authors in various contexts \cite{CC10}, \cite{Ru03}, and the problem of secant defectiveness has been widely studied for Segre-Veronese varieties, Grassmannians, Lagrangian Grassmannians, Spinor varieties and flag varieties \cite{AH95}, \cite{AB13}, \cite{AOP09}, \cite{CGG05}, \cite{CGG08},\cite{CGG11}, \cite{LP13}, \cite{BBC12}, \cite{BCC11}, \cite{MR19}, \cite{AMR19}, \cite{FMR20}, \cite{FCM19}.

An important concept related to the theory of secant varieties is that of \textit{identifiability}. We say that a point $p\in\mathbb{P}^N$ is $h$-identifiable, with respect to a non-degenerated variety $X\subseteq\mathbb{P}^N$, if it lies on a unique $(h-1)$-plane in $\mathbb{P}^N$ that is $h$-secant to $X$. Especially when $\mathbb{P}^N$ can be interpreted as a tensor space, identifiability and tensor decomposition algorithms are central in applications for instance in biology, Blind Signal Separation, data compression algorithms, analysis of mixture models psycho-metrics, chemometrics, signal processing, numerical linear algebra, computer vision, numerical analysis, neuroscience and graph analysis \cite{DD13a}, \cite{DD13b}, \cite{DD15}, \cite{KAD11}, \cite{SB00}, \cite{BK09}, \cite{CGLM08}, \cite{LO15}, \cite{MR13}. In pure mathematics identifiability questions often appears in rationality problems \cite{MM13}, \cite{Ma16}.   

Let $SV^{n_1,\dots,n_r}_{d_1,\dots,d_r}$ be the Segre-Veronese variety given as the image, in $\mathbb{P}^N$ with $N = \prod_{i=1}^r\binom{n_i+d_i}{d_i}-1$, of $\mathbb{P}^{n_1}\times\dots\times\mathbb{P}^{n_r}$ under the embedding induced by the complete linear system $\big|\mathcal{O}_{\mathbb{P}^{n_1}\times\dots\times\mathbb{P}^{n_r}}(d_1,\dots, d_r)\big|$.

For Segre-Veronese varieties the problem of secant defectiveness has been solved in some very special cases, mostly for products of few factors \cite{CGG05}, \cite{AB09}, \cite{Ab10}, \cite{BCC11}, \cite{AB12}, \cite{BBC12}, \cite{AB13}. Secant defectiveness for Segre-Veronese products  $\mathbb{P}^{1}\times\dots\times\mathbb{P}^{1}$, with arbitrary number of factors and degrees, was completely settled in \cite{LP13}. In general, $h$-defectiveness is classified only for small values of $h$ \cite[Proposition 3.2]{CGG05}. Moreover, results on the identifiability of Segre-Veronese varieties have been recently given in \cite{FCM20}, and in \cite{BBC18} under hypotheses on non secant defectiveness. In general, $h$-defectiveness of Segre products $\mathbb{P}^{n_1}\times\dots\times\mathbb{P}^{n_r}\subseteq\mathbb{P}^N$ is classified only for $h\leq 6$ \cite{AOP09}.

Our main results on non secant defectiveness and identifiability of Segre-Veronese varieties in Theorem \ref{main_SV} and Corollary \ref{Id_SV} can be summarized as follows.

\begin{thm}\label{main2}
The Segre-Veronese variety $SV^{n_1,\dots,n_r}_{d_1,\dots,d_r}\subseteq\mathbb{P}^N$ is not $h$-defective for 
$$h\leq \frac{d_j}{n_j+d_j}\frac{1}{1+\sum_{i=1}^r n_i}\prod_{i=1}^r \binom{n_i+d_i}{d_i}$$
where $\frac{n_j}{d_j} = \max_{1\leq i\leq r}\left\lbrace\frac{n_i}{d_i}\right\rbrace$. Furthermore, if in addition $2\sum_{i=1}^r n_i < \frac{d_j}{n_j+d_j}\frac{1}{1+\sum_{i=1}^r n_i}\prod_{i=1}^r \binom{n_i+d_i}{d_i}$, under the bound above, $SV^{n_1,\dots,n_r}_{d_1,\dots,d_r}\subseteq\mathbb{P}^N$ is $(h-1)$-identifiable.
\end{thm}

Note that Theorem \ref{main2} gives a polynomial bound of degree $\sum_i n_i$ in the $d_i$, while in the $n_i$ we have a polynomial bound of degree $\sum_id_i-2$. For Segre-Veronese varieties, the expected generic rank is given by a polynomials of degree $\sum_i n_i$ in the $d_i$ and of degree $\sum_i d_i - 1$ in the $n_i$. At the best of our knowledge, the bound in Theorem \ref{main2} is the best general bound so far for non secant defectiveness and identifiability of Segre-Veronese varieties.  

\subsection*{Organization of the paper} The paper is organized as follows. In Section \ref{triang} we introduce a technique to study secant defectiveness based on polytope triangulations . In Section \ref{sec:2sec} we classify possibly singular $2$-secant defective toric surfaces. In this context we would like to mention that $2$-secant defective smooth toric varieties were classified in \cite{CS07}. In Section \ref{sec:boundsSV} we prove Theorem \ref{main2} and in Proposition \ref{prop:twofactors} we recover, with our techniques, a previously known classification of some special secant defective two factors Segre-Veronese varieties. In Section \ref{Magma} we discuss a Magma \cite{Magma97} implementation of our techniques. Finally, in Section \ref{sec:def} we give new examples of defective Segre-Veronese varieties.

\subsection*{Acknowledgments}
The first named author was partially supported by Proyecto FONDECYT Regular N. 1190777. The second named author is a member of the Gruppo Nazionale per le Strutture Algebriche, Geometriche e le loro Applicazioni of the Istituto Nazionale di Alta Matematica "F. Severi" (GNSAGA-INDAM).

We thank very much G. Ottaviani for suggesting us the proof of Proposition \ref{Giorgio} and for many helpful comments on a preliminary version of the paper, and J. Draisma for many useful discussions.

\section{A convex geometry translation of Terracini's lemma}\label{triang}
Let $N$ be a rank $n$ free abelian group, $M := {\rm Hom}(N,\mathbb Z)$ its dual and $M_{\mathbb Q} := M\otimes_{\mathbb Z}\mathbb Q$ the corresponding rational vector space. Let $P\subseteq M_{\mathbb Q}$ be a full-dimensional lattice polytope, that is the convex hull of finitely many points in $M$ which do not lie on a hyperplane. The polytope $P$ defines a polarized pair $(X_P,H)$ consisting of the toric variety $X_P$ together with a very ample Cartier 
divisor $H$ of $X_P$. More precisely $X_P$ is the Zariski closure of the image of 
the monomial map
$$
\begin{array}{lccc}
\phi_P: & (\mathbb C^*)^n & \longrightarrow & \mathbb{P}^{|P\cap M|-1}\\
& u & \mapsto & [\chi^m(u)\, :\, m\in P\cap M]
\end{array}
$$
where $\chi^m(u)$ denotes the Laurent monomial in the variables $(u_1,\dots,u_n)$ defined by the point $m$, and $H$ is a hyperplane section of $X_P$. The defining fan $\Sigma := \Sigma(X)\subseteq N_{\mathbb Q}$ 
of the normalization $\tilde X_P$ of $X_P$ is the normal fan of $P$ and $H = -\Sigma_{\rho\in\Sigma(1)}\min_{m\in P}\langle m,\rho\rangle D_\rho$, where each $\rho$ denotes the primitive generator of a $1$-dimensional cone of $\Sigma$ and $D_\rho$ is the corresponding torus invariant divisor. Each element $v\in N$ defines a $1$-parameter subgroup of $(\mathbb C^*)^n$ via the homomorphism $\eta_v\colon \mathbb C^*\to (\mathbb C^*)^n$ defined by $t\mapsto t^v$. We denote by $\Gamma_v\subseteq X$ the Zariski closure of the curve $(\phi_P\circ\eta_v)(\mathbb C^*)$.

Given $a\in\mathbb C^*$ denote by 
$\Gamma_v(a)$ the point $\phi_P(\eta_v(a))$, and by $m_1,\dots,m_r$ the integer points of $P\cap M$.

\begin{Lemma}
\label{lem:tang}
Given a point $a\in\mathbb C^*$ 
the tangent space of $X$ at $\Gamma_v(a)$
is the projectivization of the vector subspace 
of $\mathbb C^{|P\cap M|}$ generated by 
the rows of the following matrix
$$
 M_v(a) :=\left(
 \begin{array}{ccc}
  a^{\langle m_1,v\rangle} & \dots & a^{\langle m_r,v\rangle}\\
  \langle m_1,e_1\rangle a^{\langle m_1-e_1^*,v\rangle} 
  & \dots 
  & \langle m_r,e_1\rangle a^{\langle m_r-e_1^*,v\rangle}\\
  \vdots & & \vdots\\
  \langle m_1,e_n\rangle a^{\langle m_1-e_n^*,v\rangle} 
  & \dots 
  & \langle m_r,e_n\rangle a^{\langle m_r-e_n^*,v\rangle}
 \end{array}\right)
$$
\end{Lemma}
\begin{proof}
The point $\Gamma_v(a)$ is in the 
image of $\phi_P$, so that we can 
use this parametrization to compute
the tangent space. Observe that since
$P$ is full-dimensional, the map $\phi_P$
is finite, moreover it is \'etale being 
equivariant with respect to the torus
action. It follows that $\phi_P$ is
smooth and thus the tangent space
of $X$ at $\Gamma_v(a)$ is spanned
by the partial derivatives of order less than or equal to one
of the monomials $\chi^{m_1}(u),\dots,\chi^{m_r}(u)$
evaluated at $a^v$.
\end{proof}

\begin{Remark}
Given a subset $\Delta := \{m_{i_0},\dots,m_{i_n}\}$
of cardinality $n+1$ of $P\cap M$ the 
corresponding $(n+1)\times (n+1)$ minor 
of the matrix $M_v(a)$, whenever $a\neq 0$, is 
\[
 \delta_{v,\Delta}(a) := 
 \frac{a^{\langle m_{i_0}+\dots+m_{i_n},v\rangle}}
 {a^{\langle e_1^*+\dots+e_{n}^*,v\rangle}}
 \left|
 \begin{matrix}
  1 & \dots & 1\\
  \langle m_{i_0},e_1\rangle 
  & \dots 
  & \langle m_{i_n},e_1\rangle \\
  \vdots & & \vdots\\
  \langle m_{i_0},e_n\rangle
  & \dots 
  & \langle m_{i_n},e_n\rangle
 \end{matrix}
 \right|
\]
Observe that $\delta_{v,\Delta}(a)$ is 
non-zero exactly when the points
of $\Delta$ do not lie on a hyperplane.
\end{Remark}
Our strategy now is to consider 
vectors $v_1,\dots,v_k\in N$, not
necessarily primitive, and study when
the linear span $\Lambda_{v_1,\dots,v_k}(a)$
of the tangent spaces
of $X$ at the points $\Gamma_{v_1}(a),
\dots,\Gamma_{v_k}(a)$ has the expected 
dimension. By Lemma~\ref{lem:tang}
the space $\Lambda_{v_1,\dots,v_k}(a)$
is the linear span of the vertical join 
$M_{v_1,\dots,v_k}(a)$ of the matrices 
$M_{v_1}(a),\dots,M_{v_k}(a)$.
Given a set $\Delta$
of cardinality $n+1$ we denote by
\stepcounter{thm}
\begin{equation}\label{bary}
b(\Delta)
 := 
 \frac{1}{n+1}\sum_{m\in\Delta}m
\end{equation}
its barycenter.

We will need the following result \cite{Ja08}. Given $K,L\subseteq [n]=\{1,2,\dots,n\}$ and a $n\times n$ matrix $A$ we denote by $A_{K,L} $ the determinant of the submatrix obtained from $A$ whose rows and columns are indexed by the set $K$ and $L$ respectively. 
\begin{Proposition}[Laplace's generalized expansion for the determinant]\label{Laplace}
Let $A$ be an $n\times n$ matrix, $m<n$ a positive integer and fix a set of rows $J=\{j_1<\dots <j_m\}$. Then
$$\det(A)=\sum_{I=\{i_1< \dots < i_m\}\subseteq [n]} (-1)^{i_1+\dots+i_m+j_1+\dots+ j_m} A_{J,I}A_{J',I'}$$
where $I'=[n]\setminus I$ and $J'=[n] \setminus J$. 
\end{Proposition}

The following is the main technical tool in our strategy.

\begin{Proposition}\label{pro:rank}
Let $S$ be a subset of $P\cap M$ and assume the
following.
\begin{enumerate}
\item
There are disjoint subsets $S_1,\dots, S_k$
of $S$ of cardinality $n+1$ each of 
which is not contained in a hyperplane.
\item
There are $v_1,\dots,v_k\in N$ such that
for each $1\leq i\leq k$ and each subset $\Delta$ not contained in a hyperplane
of cardinality $n+1$ of $S\:\setminus
\: S_1\cup\dots\cup S_{i-1}$, the value
$\langle b(\Delta),v_i\rangle$
attains its maximum exactly 
at $\Delta = S_i$.
\end{enumerate}
Then, up to a rescaling of the $v_i$ if needed, the matrix $M_{v_1,\dots,v_k}(a)$
has maximal rank $(n+1)k$ for any $a$ big enough.

Moreover, if in addition $S=P\cap M$ and $P\cap M \:\setminus \:S_1\cup\dots\cup S_{k}$ is affinely independent then the matrix $M_{v_1,\dots,v_k,v_{k+1}}(a)$
has maximal rank $|P\cap M|$ for any $a$ big enough and any vector $v_{k+1}\neq 0$.
\end{Proposition}
\begin{proof}
First of all observe that the rank 
of $M_{v_1,\dots,v_k}(a)$ does not
change if we multiply one of its rows
by a non zero constant. We apply this
modification to the matrix by multiplying
the $(i+1)$-th row of $M_{v}(a)$ by
$a^{\langle e_i^{*},v\rangle}$ for $i=1,\dots,n$. In this way
for each subset $\Delta:= \{m_{i_0},\dots,m_{i_n}\}
\subseteq S$ of cardinality $n+1$ the minor 
$\delta_{v,\Delta}(a)$ becomes
\[
 \widetilde\delta_{v,\Delta}(a) := 
  a^{(n+1)\langle b(\Delta),v\rangle}
 \left|
 \begin{matrix}
  1 & \dots & 1\\
  \langle m_{i_0},e_1\rangle 
  & \dots 
  & \langle m_{i_n},e_1\rangle \\
  \vdots & & \vdots\\
  \langle m_{i_0},e_n\rangle
  & \dots 
  & \langle m_{i_n},e_n\rangle
 \end{matrix}
 \right|
\]
Let $\widetilde M_{v_1,\dots,v_k}(a)$ be the
modified matrix and let $\widetilde M$ be 
the $(n+1)\times k$ square submatrix
whose columns correspond to the points 
of the set $S$.

We denote by $\mathcal P(n+1,k)$
the set of partitions of $S$ into $k$
disjoint subsets of cardinality $n+1$.
The determinant of $\widetilde M$ is a Laurent
polynomial in $\mathbb C[a^{\pm 1}]$ with exponents given by sums of $k$ terms of the form $(n+1)\left\langle b(\Delta),v_i\right\rangle$. 
Applying the Laplace's expansion in Proposition \ref{Laplace} several times we can write this determinant as follows:
$$ 
 \mbox{det}(\widetilde M)
 = 
 \sum_{(I_1,\dots,I_k)\in\mathcal P(n+1,k)} 
 \mbox{sign} (I_1,\dots,I_k) M_{I_1}M_{I_2}\cdots M_{I_k}
$$
where $$\mbox{sign} (I_1,\dots,I_k)=(-1)^{1+2+\dots+(k-1)(n+1)+\sum_{j\in I_1\cup \dots \cup I_{k-1}}{j}}$$ and
$M_{I_j}$ is the determinant
of the $(n+1)\times(n+1)$ submatrix
of $\widetilde M$ whose columns and rows are labeled respectively by $I_j$ and $\{(j-1)(n+1)+1,\dots, j(n+1)\}$.

By the first assumption in the hypothesis,
one of its terms is the non zero product
\[
 \widetilde\delta_{v_1,S_1}(a)\cdots\widetilde\delta_{v_k,S_k}(a)
\]
Moreover, observe that each term of the determinant
has the above form for some partition of
$S$ into $k$ disjoint subsets of cardinality
$n+1$.
We will show that, up to rescaling
the vectors $v_1,\dots,v_k$, the above 
product is the leading term of the determinant
and thus the matrix has maximal rank.
By the second assumption in the hypothesis
the degree of $\widetilde\delta_{v_1,S_1}(a)$
is bigger than the degree of 
$\widetilde\delta_{v_1,\Delta}(a)$
for any $\Delta\neq S_1$. Multiplying 
$v_1$ by a positive integer we can
also assume that the degree of 
$\widetilde\delta_{v_1,S_1}(a)$
is bigger than the degree of 
$\widetilde\delta_{v_j,\Delta}(a)$
for any $j>1$ and any $\Delta\subseteq S$
of cardinality $n+1$.
In a similar way one proves inductively 
that, up to re-scaling $v_i$, the following 
inequalities hold
\[
 \deg\widetilde\delta_{v_i,S_i}(a)
 >
 \begin{cases}
  \deg\widetilde\delta_{v_i,\Delta}(a) 
  & \text{for any $\Delta\subseteq  
  S\setminus S_1\cup\dots\cup S_{i-1}$} \text{ (by hypothesis (2))}\\
  \deg\widetilde\delta_{v_j,\Delta}(a)
  &\text{for any $j>i$ and any $\Delta\subseteq S$ (taking a big enough multiple of $v_i$)}
 \end{cases}
\]
Note that we can choose the $v_i$ all distinct.
The statement now follows by comparing
the degree of $\widetilde\delta_{v_1,S_1}(a)
\cdots\widetilde\delta_{v_k,S_k}(a)$ with the
degree of any other term of the determinant 
coming from a different partition of $S$.

Finally, if in addition $S=P\cap M$ and $P\cap M\:\setminus \:S_1\cup\dots\cup S_{k}$ is affinely independent the matrix $M_{v_1,\dots,v_k,v_{k+1}}(a)$
has rank at most $r$ for any $a\neq 0$ and any vector $v_{k+1}\neq 0$ since this is the dimension of the subspace spanned by its rows. Now, consider $S_{k+1}:=P\cap M\:\setminus \:S_1\cup\dots\cup S_{k}=\{ m_{j_1},\dots,m_{j_s}\}$ and $s=r-(n+1)k.$ Since $S_{k+1}$ is affinely independent there are $k_1,\dots,k_{s-1}$ such that the $s\times s$ matrix
$$
N= \left(
 \begin{array}{ccc}
  1 & \dots & 1\\
  \langle m_{j_1},e_{k_1}\rangle 
  & \dots 
  & \langle m_{j_s},e_{k_1}\rangle \\
  \vdots & & \vdots\\
  \langle m_{j_1},e_{k_{s-1}}\rangle
  & \dots 
  & \langle m_{j_s},e_{k_{s-1}}\rangle
 \end{array}\right)
$$
has rank $s$. Consider the submatrix
$$
 N_{v_{k+1}}(a) :=\left(
 \begin{array}{ccc}
  a^{\langle m_{j_1},v\rangle} & \dots & a^{\langle m_{j_s},v\rangle}\\
  \langle m_{j_1},e_{k_1}\rangle a^{\langle m_{j_1}-e_{k_1}^*,v\rangle} 
  & \dots 
  & \langle m_{j_s},e_{k_1}\rangle a^{\langle m_{j_s}-e_{k_1}^*,v\rangle}\\
  \vdots & & \vdots\\
  \langle m_{j_1},e_{k_{s-1}}\rangle a^{\langle m_{j_1}-e_{k_{s-1}}^*,v\rangle} 
  & \dots 
  & \langle m_{j_s},e_{k_{s-1}}\rangle a^{\langle m_{j_s}-e_{k_{s-1}}^*,v\rangle}
 \end{array}\right)
$$
of $M_{v_{k+1}}(a)$ obtained from $M_{v_{k+1}}(a)$ taking only the rows  $1,k_1+1,\dots, k_{s-1}+1$. Now, we repeat this construction using $N_{v_{k+1}}(a)$ instead of $M_{v_{k+1}}(a)$ and obtain a $r\times r$ matrix with non zero determinant. Since it is a submatrix of $M_{v_1,\dots,v_k,v_{k+1}}(a)$ we conclude that $M_{v_1,\dots,v_k,v_{k+1}}(a)$ has rank $r$.
\end{proof}

The following is inspired by Proposition \ref{pro:rank}.

\begin{Definition}\label{def:sep}
We say that $\Delta\subseteq M$ is a simplex if $\Delta$ contains $n+1$ integer points and it is not contained in an affine hyperplane. For any vector $v\in N$ consider the linear form $\varphi_v:M\to\mathbb{R}$ given by $\varphi_v(p)=\langle p,v\rangle$. We will write 
$$\varphi_v(\Delta)=\frac{1}{n+1}\sum_{p\in \Delta}\varphi_v(p)$$
We say that $v$ separates the simplex $\Delta$ in a subset $S\subseteq M$ if 
$$\max \{\varphi_{v}(T); T\subseteq S \: | \: T \: \mbox{is a simplex} \} = \varphi_v(\Delta)$$
and the maximum is attained only at $\Delta$.
\end{Definition}

\begin{Remark}\label{remark:equalvs}
Let $P\subseteq M_{\mathbb{Q}}$ be a full-dimensional lattice polytope, $\Delta_1,\dots,\Delta_k$ disjoint simplexes contained in $P\cap M$ and $v_1,\dots,v_k$ vectors in $N$. Assume that $v_i$ separates $\Delta_i$ in $\Delta_i\cup \dots \cup \Delta_k$ for any $i=1\dots k$. 
Since the maximal in \ref{def:sep} is strict, if we take vectors $w_1,\dots,w_k$ in $N$ close enough to the $v_i$, then $w_i$ separates $\Delta_i$ in $\Delta_i\cup \dots \cup \Delta_k$ for any $i=1\dots k$. Therefore, we may assume without loss of generality that the $v_i$ are distinct. Observe that if $v_i$ separates $\Delta_i$ in $\Delta_i\cup \dots \cup \Delta_k$ for any $i=1\dots k$, then any multiple of the $v_i$ will do so as well. 
\end{Remark}

As a consequence of Proposition \ref{pro:rank} we get the following criterion for non secant defectiveness of toric varieties.

\begin{thm}\label{teo:main}
Let $P\subseteq M_{\mathbb{Q}}$ be a full-dimensional lattice polytope, 
$X_P$ the corresponding toric variety, $\Delta_1,\dots,\Delta_k$ disjoint simplexes contained in $P\cap M$ and $v_1,\dots,v_k$ vectors in $N$. If $v_i$ separates $\Delta_i$ in $\Delta_i\cup \dots \cup \Delta_k$ for any $i=1\dots k$ then $X_P$ is not $k$-defective. Moreover, if $(P\cap M) \:\setminus\: \Delta_1,\dots,\Delta_k$ is affinely independent then $X_P$ is not defective. In particular, if $P\cap M = \Delta_1\cup \dots \cup \Delta_k$ then $X_P$ is not defective.
\end{thm}
\begin{proof}
Without loss of generality we can assume that $P$ is contained in the positive quadrant and contains the origin. Applying Proposition \ref{pro:rank} with $S = \Delta_1\cup\dots\cup \Delta_k$ we get that $M_{v_1,\dots,v_k}(a)$ has maximal rank for any $a$ big enough and the $v_i$ are distinct, taking multiples if necessary. Take any $a$ big enough, then the tangent spaces of $X_P$ at the points $\Gamma_{v_1}(a),\dots, \Gamma_{v_k}(a)$ are in general position. By Terracini's lemma \cite{Te11} we conclude that $X_P$ is not $k$-defective. For the second statement just use second part of Proposition~\ref{pro:rank}.
\end{proof}

Theorem \ref{teo:main} in principle can be applied to any toric variety, in particular to Segre-Veronese varieties, one just need to describe the vectors $v_1,\dots, v_k$. Due to its recursive nature Theorem \ref{teo:main} can be algorithmically implemented, we refer to Section \ref{Magma} for a Magma implementation. 

%\begin{Example}
For instance by Theorem \ref{teo:main} we have that $SV^{(1,1)}_{(5,3)}$ is not defective.
%\end{Example}

Now, we prove two technical lemmas  in order to get a general bound for non secant defectiveness of toric varieties from Theorem \ref{teo:main}. In Section \ref{sec:boundsSV} we will specialize this bound to Segre-Veronese varieties.

\begin{Definition}
Given a finite subset $S\subseteq M$,
the {\em barycentric polytope} of $S$,
denoted by $B(S)\subseteq M_{\mathbb Q}$,
is the convex hull of all the points 
$b(\Delta)$, where $\Delta$ varies
among all the subsets of $S$
of cardinality $n+1$ which are not
contained in a hyperplane and $b(\Delta)$ is as in (\ref{bary}).
\end{Definition}

\begin{Example} 
Consider 
$$S=\{A=(0,0),B=(1,0),C=(2,0),D=(1,1),E=(2,1)\}$$
as in the picture below. We have nine possible ways to form simplexes $\Delta \subseteq S$
and the barycentric polytope $B(S)$ is a trapezoid.  In the picture we draw circles in the barycenters of simplexes $\Delta$ with $D,E\in \Delta$; we draw $+$ on barycenters of simplexes with $E\in \Delta$ but $D\notin \Delta$, and finally we draw $\times$ in barycenters of simplexes with $D\in \Delta$ but $E\notin \Delta$. 
$$
\begin{tikzpicture}[xscale=0.6,yscale=0.6]
\node[black] (A) at (0,0) {$\bullet$};
\node[black] (B) at (3,0) {$\bullet$};
\node[black] (C) at (6,0) {$\bullet$};
\node[black] (D) at (3,3) {$\bullet$};
\node[black] (E) at (6,3) {$\bullet$};
\foreach \Point in {(3,2),(4,2),(5,2)}{
    \node at \Point {$\circ$};}    
\foreach \Point in {(3,1),(4,1),(5,1)}{
    \node at \Point {$+$};}    
\foreach \Point in {(2,1),(3,1),(4,1)}{
    \node[thick] at \Point {$\times$};} 
\draw (A) node[left] {$A$};
\draw (B) node[left] {$B$};
\draw (C) node[right] {$C$};
\draw (D) node[left] {$D$};
\draw (E) node[right] {$E$};        
%\draw[thin] (2,1) -- (5,1) -- (5,2) -- (3,2) -- cycle;
\end{tikzpicture}
$$
There are exactly two shared barycenters, corresponding to the pairs of simplexes 
\[
\Delta=\{A,D,C\} \mbox{ and } \Delta':= \{A,B,E\}
\]
\[
\Delta=\{B,C,D\} \mbox{ and } \Delta':= \{A,C,E\}
\]
Note that neither of these shared barycenters are vertexes of $B(S)$. The next lemma shows that this is always the case.
\end{Example}

\begin{Lemma}
\label{lem:uniq}
Let $\Delta,\Delta'$ be two simplexes in 
$S\subseteq M$. If $b(\Delta) = b(\Delta')$ then
$b(\Delta)$ is not a vertex of $B(S)$.
\end{Lemma}
\begin{proof}
Let $\Delta = \{p_1,\dots,p_{n+1}\}$ and $\Delta' = \{p_1',\dots,p_{n+1}'\}$. We say that the pair $(p_i,p_j')$ is good if 
$$
 \Delta_{ij} := (\Delta\setminus\{p_i\})\cup\{p_j'\} \mbox{ and }
 \Delta'_{ij} := (\Delta'\setminus\{p_j'\})\cup\{p_i\}
$$
are simplexes. Observe that it is enough to show that there exists a good pair with
$\Delta_{ij}\neq\Delta$, since in this case $b(\Delta) = b(\Delta')$ is the mid point of 
the segment with vertexes $b(\Delta_{ij})$ and $b(\Delta_{ij}')$. To show the existence of a good pair
let us denote by $\Lambda_i$ the hyperplane spanned by $\Delta\setminus\{p_i\}$ and by $\Lambda_i'$ the hyperplane spanned by $\Delta'\setminus\{p_i'\}$.

Note that if either $p_i\in \Lambda_j'$ or $p_j'\in \Lambda_i$ then the pair $(p_i,p_j')$ is not good
and viceversa. Assume $p_1\notin\Delta'$. We now show that at least one pair $(p_1,p_i')$ is good. Indeed, assume the contrary, then we can partition the set $\{1,\dots,n+1\}$ into a disjoint union
$I\cup J$ of two subsets such that $p_1\in\Lambda_j'$ for any $j\in I$ and $p_i'\in \Lambda_1$ for any $i\in I$. Then we would get
$$p_1\in
 \bigcap_{j\in J} \Lambda_j'
 = \langle p_i'\, :\, i\in I\rangle
 \subseteq \Lambda_1
$$
a contradiction.
\end{proof}

\begin{Lemma}\label{lemma2}
Let $S\subseteq P\cap M$ be a subset not contained in a hyperplane. Then there exists a vector in $N$, with non-negative entries, separating a simplex in $S$.
\end{Lemma}
\begin{proof}
Without loss of generality we can assume that $P$ is contained in the positive quadrant. First, assume that there is a vertex $b(\Delta)$ of $b(S)$ whose $i$-th coordinate is strictly bigger than those of the other vertexes of $b(S)$. In this case we may simply take $v = e_{i}^{*}$. Now, if there are several vertexes with the same $i$-th coordinate, say for $i = 1$, then among these we check if there is only one maximizing the $2$-th coordinate. If so we choose $v = ae_1^{*}+e_{2}^{*}$ with $a\gg 0$. If not among the vertexes maximizing also the $2$-coordinate we consider those maximizing the $3$-th coordinate. As before we have two cases, in the first we take $v = ae_1^{*}+be_{2}^{*}+e_{3}^{*}$ with $a\gg b\gg 0$, while in the second case among these vertexes we consider those maximizing the $4$-th coordinate. Proceeding recursively in this way and noting that a vertex of $b(S)$ corresponds to a, unique by Lemma \ref{lem:uniq}, barycenter of a simplex in $S$, we get the claim. 
\end{proof}

We provide a bound for non secant defectiveness of the projective toric variety $X$ associated to a polytope $P$ by counting the maximum number of integer points on a hyperplane section of $P$.

\begin{thm}\label{hyperprop}
Let $P\subseteq M_{\mathbb{Q}}$ be a 
full-dimensional lattice polytope, 
$X_P\subseteq\mathbb P^{|P\cap M|-1}$ 
the corresponding $n$-dimensional toric variety, and $m$ the maximum number of integer points in a hyperplane section of $P$. If
$$h\leq\dfrac{|P\cap M|-m}{n+1}$$
then $X_P$ is not $h$-defective. 
\end{thm}
\begin{proof}
Set $S := P\cap M$. By Lemma \ref{lemma2} there is a vector $v_1\in N$ separating a simplex $\Delta_1$ in $P$. Now, consider $S\setminus \Delta_1$. If $|S\setminus \Delta_1|> m$ then $S$ is not contained in a hyperplane and we may apply again Lemma \ref{lemma2} to get a second vector $v_2\in N$ separating a simplex $\Delta_2$ in $S\setminus \Delta_1$. Proceeding recursively in this way, as long as $|S\setminus (\Delta_1\cup\dots\cup \Delta_{k})| > m$, we get the statement by Theorem \ref{teo:main}.
\end{proof}

In order to apply Theorem \ref{hyperprop} in specific cases we will make use of the following result asserting that the maximum number of integer points of $P$ lying on a hyperplane is attained on a facet.

\begin{Proposition}\label{hyp}
Let $P\subseteq M_{\mathbb Q}$ be full-dimensional lattice polytope such that there exist linearly independent $v_1,\dots,v_n\in N$ and facets $F_1,\dots,F_n$ such that for any $i$ we have $v_j(F_i\cap M) = v_j(P\cap M)$ for any $j\neq i$. Then, given a hyperplane $H\subseteq M_{\mathbb Q}$, there exists a facet $F_i$, with $1\leq i\leq n$, such that $|H\cap P\cap M|\leq |F_i\cap M|$.
\end{Proposition}
\begin{proof} 
Consider the map
$$
\begin{array}{cccc}
\pi_i: & M_{\mathbb Q} & \longrightarrow & \mathbb Q^{n-1}\\
 & x & \mapsto & (v_1(x),\dots,v_{i-1}(x),v_{i+1}(x),\dots,v_n(x))
\end{array}
$$
Note that by hypothesis $\pi_i(F_i\cap M) = \pi_i(P\cap M)$. Observe that there exists an index $i$ such that the restriction
of $\pi_i$ to $H$ is injective. Then $|H\cap P\cap M| = |\pi_i(H\cap P\cap M)| \leq |\pi_i(F_i\cap M)|\leq |F_i\cap M|$.
\end{proof}

\subsection{An alternative proof of Theorem \ref{hyperprop}}
The bound in Theorem \ref{hyperprop} is, at the best of our knowledge, the first general bound for non secant defectiveness of toric varieties appearing in the literature. A machinery based on tropical geometry was introduced to study secant defectiveness by J. Draisma in \cite{Dr08}. In order to use this tropical technique one has to produce a regular partition of the polytope $P$ that is a subdivision into polyhedral cones such that none of the integer points of $P$ lies on the boundaries. We thank J. Draisma for explaining this to us. In this section we give another proof of Theorem \ref{hyperprop} based on Draisma's tropical approach. We thank D. Panov for suggesting us the proof of the following result which is the first step toward the alternative proof of Theorem \ref{hyperprop}.

\begin{Lemma}\label{Panov}
Let $P\subset\mathbb{R}^n$ be a convex lattice polytope. There exist a lattice simplex $\Delta\subset P$ and an affine hyperplane $H\subset\mathbb{R}^n$ separating $\Delta$ from the convex hull of the integer points of $P\setminus \Delta$. This is equivalent to say that there exists a degree one polynomial $h:\mathbb{R}^n\rightarrow\mathbb{R}$ that is positive on all the integer points of $\Delta$ and negative on all the integer points of $P\setminus \Delta$.
\end{Lemma}
\begin{proof}
We proceed by induction on $n$. Assume that the statement holds for all the polytopes of dimension at most $n-1$. 

Let $v\in P$ be a vertex, and denote by $v_1,\dots,v_m$ the end-points of the edges of $P$ starting at $v$. Let $P'$ be the convex hull of $v,v_1,\dots,v_m$ and $P''$ the convex hull of the integer points of $P'$ except $v$. Note that $v\notin P''$. Consider a facet of $P''$ that can be connected with $v$ by a segment that does not intersect the interior of $P''$. If  $P''$ has dimension $n-1$ then we can take the whole $P''$ as such a facet. Let $H$ be the hyperplane containing this face. Then $H$ intersects only the edges of $P$ that are adjacent to $v$.

Now, cut $P$ along $H$ and denote by $Q$ the part that contains $v$, and by $F$ the face of $Q$ that lies in $H$. By construction the integer points of $Q$ are the integer points of $F$ and $v$. By induction hypothesis we can cut out a simplex $\Delta'$ in $F$ by a hyperplane $H'$ of dimension $n-2$ contained in $H$. Finally, consider a hyperplane obtained by performing an infinitesimal rotation of $H$ around $H'$. Such a hyperplane separates the simplex $\Delta$ generated by $\Delta'$ and $v$ from the convex hull of the integer points of $P\setminus \Delta$.  
\end{proof}
Before stating the next result
we recall the definition of
regular subdivision of a lattice
polytope~\cite[Definition 2.2.10]{Triang}.
Let $P\subseteq \mathbb R^n$ be a
lattice polytope, $J$ the
set of indexes of the lattice
points of $P$ and $w\colon
J\to\mathbb R$ a function.
Let $P^w\subseteq \mathbb R^{n+1}$
be the convex hull of the points
$p_i^w := (p_i,w(p_i))$ for each
$i\in J$.
The {\em regular subdivision} of
$P$ produced by $w$ is the set
of projected lower faces of
$P^w$. This regular subdivision 
is denoted by $\mathscr S(P,w)$.
\begin{Lemma}\label{regular}
Let $P\subseteq \mathbb R^n$ be a
lattice polytope and let 
$\Delta\subseteq P$ be
a lattice simplex which can be 
separated from the convex hull
$P_0$ of the lattice points 
of $P\setminus \Delta$ by a hyperplane $H$. 
Then, given a regular subdivision
$\mathscr S(P_0,w_0)$ of $P_0$ there
exists a regular subdivision
$\mathscr S(P,w)$ of $P$ which contains all the polytopes in 
$\mathscr S(P_0,w_0)$ and such
that $\Delta\in\mathscr S(P,w)$.
\end{Lemma}
\begin{proof}
We denote by $J_0$ and $J$ the
indexes for the set of lattice 
points of $P_0$ and $P$ respectively.
By definition we have $J_0\subseteq J$.
We define $w\colon J\to\mathbb R$
as $w_{|J_0} := w_0$ and extend it to $J\setminus J_0$
as follows.
Let $h\colon \mathbb{R}^n
\rightarrow\mathbb{R}$ be the
function which defines $H$. After
possibly perturbing $h$ we can
assume that it takes distinct
values on the set of vertexes
$\{p_i\, :\, i\in J\setminus J_0\}$ 
of $\Delta$. As a consequence this
set is totally ordered. After 
possibly relabeling $J$ we 
can assume $J\setminus J_0 =
\{0,\dots,n\}$ and $h(p_i)
< h(p_j)$ if $i<j$ and both 
indexes are in $J\setminus J_0$.
Define $w(p_0)$ in such a way that it is bigger
than $w(p_i)$ for any $i\in J_0$.
In this way the convex hull
of $\{p_0^w\}\cup P_0^w$ contains
all the lower facets of $P_0^w$.
Now, assume that $w$ has been
defined on $p_i$ for $0\leq i 
< r$ and define $w(p_r) > 
w(p_{r-1})$ so that for each
point $(p,\alpha)$ in the convex
hull of $\{p_0^w,\dots,p_r^w\}$
and each point $(p,\beta)$ in the
convex hull of 
$\{p_0^w,\dots,p_{r-1}^w\}\cup
P_0^w$ the inequality $\alpha\geq
\beta$ holds.
By construction the convex hull
of $\{p_0,\dots,p_r\}$ is in the
latter regular subdivision.
Moreover, due to the fact that
all the points $p_0^w,\dots,
p_r^w$ have last coordinate 
bigger than those of the 
remaining lifted lattice points
of $P_0$, it follows that any
lower face of $P_0$ is in the 
latter regular subdivision.
The statement follows by
induction on $r$.
\end{proof}

\begin{Remark}
While applying the inductive
procedure to produce the new regular 
subdivision in Lemma~\ref{regular}
several new regular subdivisions can
be created and destroyed along the
way as shown in the following picture:
$$
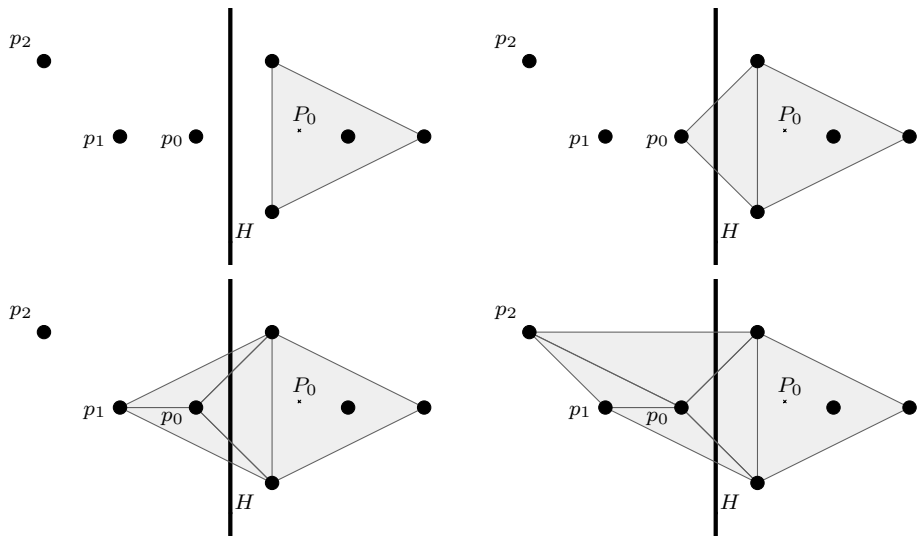
\definecolor{wqwqwq}{rgb}{0.3764705882352941,0.3764705882352941,0.3764705882352941}
\begin{tikzpicture}[line cap=round,line join=round,>=triangle 45,x=1.0cm,y=1.0cm]
\clip(1.3,0.3) rectangle (7.3,3.7);
\fill[line width=0.4pt,color=wqwqwq,fill=wqwqwq,fill opacity=0.10000000149011612] (5.,1.) -- (5.,3.) -- (7.,2.) -- cycle;
\draw [line width=0.4pt,color=wqwqwq] (5.,1.)-- (5.,3.);
\draw [line width=0.4pt,color=wqwqwq] (5.,3.)-- (7.,2.);
\draw [line width=0.4pt,color=wqwqwq] (7.,2.)-- (5.,1.);
\draw [line width=1.6pt] (4.4509302325581395,0.3) -- (4.4509302325581395,3.7);
\begin{scriptsize}
\draw [fill=black] (5.,1.) circle (2.5pt);
\draw [fill=black] (5.,3.) circle (2.5pt);
\draw [fill=black] (7.,2.) circle (2.5pt);
\draw [fill=black] (4.,2.) circle (2.5pt);
\draw[color=black] (3.69,1.94) node {$p_0$};
\draw [fill=black] (3.,2.) circle (2.5pt);
\draw[color=black] (2.67,1.94) node {$p_1$};
\draw [fill=black] (2.,3.) circle (2.5pt);
\draw[color=black] (1.7,3.26) node {$p_2$};
\draw [fill=black] (6.,2.) circle (2.5pt);
\draw [color=black] (5.36,2.08)-- ++(-0.5pt,-0.5pt) -- ++(1.0pt,1.0pt) ++(-1.0pt,0) -- ++(1.0pt,-1.0pt);
\draw[color=black] (5.43,2.28) node {$P_0$};
\draw [fill=black] (4.454418793307419,0.6000324488761196) circle (0.5pt);
\draw[color=black] (4.64,0.75) node {$H$};
\end{scriptsize}
\end{tikzpicture}
\quad
\begin{tikzpicture}[line cap=round,line join=round,>=triangle 45,x=1.0cm,y=1.0cm]
\clip(1.3,0.3) rectangle (7.3,3.7);
\fill[line width=0.4pt,color=wqwqwq,fill=wqwqwq,fill opacity=0.10000000149011612] (5.,1.) -- (5.,3.) -- (7.,2.) -- cycle;
\fill[line width=0.4pt,color=wqwqwq,fill=wqwqwq,fill opacity=0.10000000149011612] (5.,3.) -- (4.,2.) -- (5.,1.) -- cycle;
\draw [line width=0.4pt,color=wqwqwq] (5.,1.)-- (5.,3.);
\draw [line width=0.4pt,color=wqwqwq] (5.,3.)-- (7.,2.);
\draw [line width=0.4pt,color=wqwqwq] (7.,2.)-- (5.,1.);
\draw [line width=1.6pt] (4.4509302325581395,0.3) -- (4.4509302325581395,3.7);
\draw [line width=0.4pt,color=wqwqwq] (5.,3.)-- (4.,2.);
\draw [line width=0.4pt,color=wqwqwq] (4.,2.)-- (5.,1.);
\draw [line width=0.4pt,color=wqwqwq] (5.,1.)-- (5.,3.);
\begin{scriptsize}
\draw [fill=black] (5.,1.) circle (2.5pt);
\draw [fill=black] (5.,3.) circle (2.5pt);
\draw [fill=black] (7.,2.) circle (2.5pt);
\draw [fill=black] (4.,2.) circle (2.5pt);
\draw[color=black] (3.69,1.94) node {$p_0$};
\draw [fill=black] (3.,2.) circle (2.5pt);
\draw[color=black] (2.67,1.94) node {$p_1$};
\draw [fill=black] (2.,3.) circle (2.5pt);
\draw[color=black] (1.7,3.26) node {$p_2$};
\draw [fill=black] (6.,2.) circle (2.5pt);
\draw [color=black] (5.36,2.08)-- ++(-0.5pt,-0.5pt) -- ++(1.0pt,1.0pt) ++(-1.0pt,0) -- ++(1.0pt,-1.0pt);
\draw[color=black] (5.43,2.28) node {$P_0$};
\draw [fill=black] (4.454418793307419,0.60003244887612) circle (0.5pt);
\draw[color=black] (4.64,0.75) node {$H$};
\end{scriptsize}
\end{tikzpicture}
$$
$$
\definecolor{wqwqwq}{rgb}{0.3764705882352941,0.3764705882352941,0.3764705882352941}
\begin{tikzpicture}[line cap=round,line join=round,>=triangle 45,x=1.0cm,y=1.0cm]
\clip(1.3,0.3) rectangle (7.3,3.7);
\fill[line width=0.4pt,color=wqwqwq,fill=wqwqwq,fill opacity=0.10000000149011612] (5.,1.) -- (5.,3.) -- (7.,2.) -- cycle;
\fill[line width=0.4pt,color=wqwqwq,fill=wqwqwq,fill opacity=0.10000000149011612] (5.,3.) -- (4.,2.) -- (5.,1.) -- cycle;
\fill[line width=0.4pt,color=wqwqwq,fill=wqwqwq,fill opacity=0.10000000149011612] (5.,3.) -- (3.,2.) -- (4.,2.) -- cycle;
\fill[line width=0.4pt,color=wqwqwq,fill=wqwqwq,fill opacity=0.10000000149011612] (3.,2.) -- (5.,1.) -- (4.,2.) -- cycle;
\draw [line width=0.4pt,color=wqwqwq] (5.,3.)-- (7.,2.);
\draw [line width=0.4pt,color=wqwqwq] (7.,2.)-- (5.,1.);
\draw [line width=1.6pt] (4.4509302325581395,0.3) -- (4.4509302325581395,3.7);
\draw [line width=0.4pt,color=wqwqwq] (5.,3.)-- (4.,2.);
\draw [line width=0.4pt,color=wqwqwq] (4.,2.)-- (5.,1.);
\draw [line width=0.4pt,color=wqwqwq] (5.,1.)-- (5.,3.);
\draw [line width=0.4pt,color=wqwqwq] (5.,3.)-- (3.,2.);
\draw [line width=0.4pt,color=wqwqwq] (3.,2.)-- (4.,2.);
\draw [line width=0.4pt,color=wqwqwq] (4.,2.)-- (5.,3.);
\draw [line width=0.4pt,color=wqwqwq] (3.,2.)-- (5.,1.);
\draw [line width=0.4pt,color=wqwqwq] (5.,1.)-- (4.,2.);
\draw [line width=0.4pt,color=wqwqwq] (4.,2.)-- (3.,2.);
\begin{scriptsize}
\draw [fill=black] (5.,1.) circle (2.5pt);
\draw [fill=black] (5.,3.) circle (2.5pt);
\draw [fill=black] (7.,2.) circle (2.5pt);
\draw [fill=black] (4.,2.) circle (2.5pt);
\draw[color=black] (3.69,1.85) node {$p_0$};
\draw [fill=black] (3.,2.) circle (2.5pt);
\draw[color=black] (2.67,1.94) node {$p_1$};
\draw [fill=black] (2.,3.) circle (2.5pt);
\draw[color=black] (1.7,3.26) node {$p_2$};
\draw [fill=black] (6.,2.) circle (2.5pt);
\draw [color=black] (5.36,2.08)-- ++(-0.5pt,-0.5pt) -- ++(1.0pt,1.0pt) ++(-1.0pt,0) -- ++(1.0pt,-1.0pt);
\draw[color=black] (5.43,2.28) node {$P_0$};
\draw [fill=black] (4.454418793307419,0.60003244887612) circle (0.5pt);
\draw[color=black] (4.64,0.75) node {$H$};
\end{scriptsize}
\end{tikzpicture}\quad
\begin{tikzpicture}[line cap=round,line join=round,>=triangle 45,x=1.0cm,y=1.0cm]
\clip(1.3,0.3) rectangle (7.3,3.7);
\fill[line width=0.4pt,color=wqwqwq,fill=wqwqwq,fill opacity=0.10000000149011612] (5.,1.) -- (5.,3.) -- (7.,2.) -- cycle;
\fill[line width=2.pt,color=wqwqwq,fill=wqwqwq,fill opacity=0.10000000149011612] (5.,3.) -- (4.,2.) -- (5.,1.) -- cycle;
\fill[line width=0.4pt,color=wqwqwq,fill=wqwqwq,fill opacity=0.10000000149011612] (3.,2.) -- (5.,1.) -- (4.,2.) -- cycle;
\fill[line width=0.4pt,color=wqwqwq,fill=wqwqwq,fill opacity=0.10000000149011612] (3.,2.) -- (2.,3.) -- (4.,2.) -- cycle;
\fill[line width=0.4pt,color=wqwqwq,fill=wqwqwq,fill opacity=0.10000000149011612] (2.,3.) -- (5.,3.) -- (4.,2.) -- cycle;
\draw [line width=0.4pt,color=wqwqwq] (5.,3.)-- (7.,2.);
\draw [line width=0.4pt,color=wqwqwq] (7.,2.)-- (5.,1.);
\draw [line width=1.6pt] (4.4509302325581395,0.3) -- (4.4509302325581395,3.7);
\draw [line width=0.4pt,color=wqwqwq] (5.,3.)-- (4.,2.);
\draw [line width=0.4pt,color=wqwqwq] (4.,2.)-- (5.,1.);
\draw [line width=0.4pt,color=wqwqwq] (5.,1.)-- (5.,3.);
\draw [line width=0.4pt,color=wqwqwq] (3.,2.)-- (5.,1.);
\draw [line width=0.4pt,color=wqwqwq] (5.,1.)-- (4.,2.);
\draw [line width=0.4pt,color=wqwqwq] (4.,2.)-- (3.,2.);
\draw [line width=0.4pt,color=wqwqwq] (3.,2.)-- (2.,3.);
\draw [line width=0.4pt,color=wqwqwq] (2.,3.)-- (4.,2.);
\draw [line width=0.4pt,color=wqwqwq] (4.,2.)-- (3.,2.);
\draw [line width=0.4pt,color=wqwqwq] (2.,3.)-- (5.,3.);
\draw [line width=0.4pt,color=wqwqwq] (5.,3.)-- (4.,2.);
\draw [line width=0.4pt,color=wqwqwq] (4.,2.)-- (2.,3.);
\begin{scriptsize}
\draw [fill=black] (5.,1.) circle (2.5pt);
\draw [fill=black] (5.,3.) circle (2.5pt);
\draw [fill=black] (7.,2.) circle (2.5pt);
\draw [fill=black] (4.,2.) circle (2.5pt);
\draw[color=black] (3.69,1.85) node {$p_0$};
\draw [fill=black] (3.,2.) circle (2.5pt);
\draw[color=black] (2.67,1.94) node {$p_1$};
\draw [fill=black] (2.,3.) circle (2.5pt);
\draw[color=black] (1.7,3.26) node {$p_2$};
\draw [fill=black] (6.,2.) circle (2.5pt);
\draw [color=black] (5.36,2.08)-- ++(-0.5pt,-0.5pt) -- ++(1.0pt,1.0pt) ++(-1.0pt,0) -- ++(1.0pt,-1.0pt);
\draw[color=black] (5.43,2.28) node {$P_0$};
\draw [fill=black] (4.454418793307419,0.60003244887612) circle (0.5pt);
\draw[color=black] (4.64,0.75) node {$H$};
\end{scriptsize}
\end{tikzpicture}
$$
Note that at each step the regular 
subdivision of $P_0$ is left unaltered.
\end{Remark}

\begin{thm}\label{main_Dra}
Let $P\subseteq M_{\mathbb{Q}}$ be a full-dimensional lattice polytope and $X_P$ the corresponding $n$-dimensional toric variety. Consider a regular subdivision of $P$ into $k$ open simplexes such that no integer point of $P$ lies on the boundaries. Assume that among these simplexes exactly $k_i$ are $i$-dimensional. Then
$$\dim (\Sec_k(X_P))\geq \sum_{i=0}^nk_i(i+1)-1$$
In particular, if in the regular subdivision of $P$ there are $k$ full-dimensional simplexes then $X_P$ is not $k$-defective. 
\end{thm}
\begin{proof}
In the terminology of \cite[Section 2]{Dr08} the integer points of $P$ lying in a simplex are a set of winning directions. Therefore, the statement follows from \cite[Corollary 2.3]{Dr08}.
\end{proof}

\begin{proof}[Alternative proof of Theorem \ref{hyperprop}]
Set $S := P\cap M$. By Lemma \ref{Panov} there is a hyperplane $H_1$ separating a simplex $\Delta_1$ in $P$. Now, consider $S\setminus \Delta_1$. If $|S\setminus \Delta_1|> m$ then $S$ is not contained in a hyperplane and we may apply again Lemma \ref{Panov} to get a second hyperplane $H_2$ separating a simplex $\Delta_2$ in $S\setminus \Delta_1$. Proceeding recursively in this way, as long as $|S\setminus (\Delta_1\cup\dots\cup \Delta_{k})| > m$, and applying Lemma \ref{regular} we get the statement by Theorem \ref{main_Dra}.
\end{proof}

\begin{Remark}
The main difference between our method for checking non defectivity and the tropical one 
described in Theorem \ref{hyperprop} is the following. In both methods one has to separate a lattice simplex $\Delta$ from the convex hull of the set $S$ of lattice points. In our case this means that one has to separate a vertex 
of the barycentric polytope, while in the tropical case one has to separate the lattice points in $\Delta$ from the remaining ones by means of a hyperplane. It is clear that the latter separation
implies the former but the converse is not true in general as shown by the following example. Let
$$S_1:=\{P_1,P_2,P_3\},\ S_2:=\{Q_1,Q_2,Q_3\}$$
where
$$P_1=(0,0),P_2=(3,1),P_3=(4,0),
Q_1=(-1,-2),Q_2=(1,3),Q_3=(2,2)$$
Then $v=(1,0)$ separates $S_1$ in $S_1\cup S_2$. 
However, the convex hulls of $S_1$ and $S_2$ 
overlap as shown in the following picture.
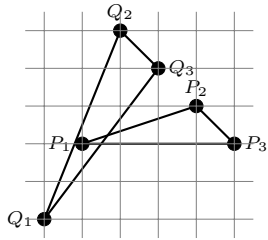
\begin{center}
 \begin{tikzpicture}[scale=.5]
  \tkzDefPoint[label=left:\tiny $P_1$](0,0){P1}
  \tkzDefPoint[label=above:\tiny $P_2$](3,1){P2}
  \tkzDefPoint[label=right:\tiny $P_3$](4,0){P3}
  \tkzDefPoint[label=left:\tiny $Q_1$](-1,-2){Q1}
  \tkzDefPoint[label=above:\tiny $Q_2$](1,3){Q2}
  \tkzDefPoint[label=right:\tiny $Q_3$](2,2){Q3}
  %\tkzFillPolygon[color = gray!50](P1,P2,P3)
 % \tkzFillPolygon[color = gray!50](Q1,Q2,Q3)
  \tkzDrawPoints[fill=black,color=black,size=5](P1,P2,P3,Q1,Q2,Q3)
  \tkzDrawSegments[thick](P1,P2 P2,P3 P3,P1 Q1,Q2 Q2,Q3 Q3,Q1)
  \draw[help lines] (-1.5,-2.5) grid (4.5,3.5);
  \end{tikzpicture}
\end{center}
In particular the convex hulls of the simplexes in Proposition \ref{pro:rank} may overlap. Our method thus starts from determining a general linear form $\phi$ on the linear span $\langle S\rangle$ and then separating the simplex whose barycenter has the biggest value with respect to $\phi$. In the tropical approach one has to check whether the $n+1$ lattice points corresponding to the biggest $n+1$ values of $\phi$ span a simplex. Otherwise another $\phi$ has to be chosen. In the above example the form corresponding to $(1,0)$ does not work with the tropical method, while the form corresponding to $(-1,1)$ gives a hyperplane separating $\{P_1,Q_2,Q_3\}$ from $\{P_2,P_3,Q_1\}$.
\end{Remark}

\section{An application to surfaces}\label{sec:2sec}
In this section we apply our methods to prove the following well-known fact.
\begin{Proposition}\label{2def} 
Let $P\subseteq M_{\mathbb{Q}}$ be a $2$-dimensional lattice polytope and $X_P$ the corresponding $2$-dimensional toric variety. Then $X_P$ is $2$-defective if and only if either $X_P$ is a cone or $P$ is contained in $V_2^2$.
\end{Proposition}
\begin{proof}
Clearly, if $X_P$ is a cone or $P$ is contained in the polytope of $V_2^2$ then $X_P$ is $2$-defective. Assume that neither $X_P$ is a cone nor $P$ is contained in the polytope of $V_2^2$. We may assume that $M=\mathbb{Z}^2$, $P$ has at least $6$ points, $A=(0,0),B=(0,1),C=(1,0)\in P$ and $P$ is contained in the first quadrant.

To simplify the notation let us write $D=(2,0),E=(1,1),F=(0,2),\Delta_0=\{A,B,C\}$. We distinguish three cases depending on how many points there are in $P\cap \{D,E,F\}$.

First assume that there are two points $p,q$ in $\{D,E,F\}\cap P$. Then there is at least one point $r\in (P\cap M)\setminus \Delta_2^2$. Hence, using $\Delta_1=\Delta_0, v_1=(-1,-1),\Delta_2=\{p,q,r\}$ and any $v_2$, we see that $X_P$ is not $2$-defective by Theorem \ref{teo:main}.

Now, assume that $\{p\}=\{D,E,F\}\cap P$ has exactly one point. Then there are at least two points $q,r\in (P\cap M)\setminus \Delta_2^2$. If there are such two points making $\Delta_2=\{p,q,r\}$ a simplex we are done as in the previous case. We therefore can assume that all points in $(P\cap M)\setminus \Delta_0$ are collinear. We will prove that $p=E$. Indeed, the points of $P\setminus \Delta_0$ can not all lie in the segment $\{(x,0),x\geq 2\}$ since $X_P$ is not a cone, and similarly they can not all lie on the segment $\{(0,y),y\geq 2\}$. Therefore, there is a point $G=(x,y)\in P$ with $x\geq 1$ and $y\geq 1$. Since $E\in \overline{BCG}$ and $P$ is convex we conclude that $E\in (P\cap M)$.

Now, either the points in $(P\cap M)\setminus \Delta_0$ are contained in the vertical line $\{(1,y),y\geq 1\}$ or 
$q=(x_q,y_q),r=(x_r,y_r)$ for some $2\leq x_q<x_r$ and $1\leq y_q<y_r.$ In the first case we may use
$$\Delta_1=\{A,B,E\}, v_1=(1,-1),\Delta_2=\{(1,3),(1,2),C\}$$
with $v_2$ arbitrary, and in the second case we may use
$$\Delta_1=\{B,q,r\}, v_1=(a,1),\Delta_2=\{A,C,E\}, \mbox{ with } a\gg 0$$
again with $v_2$ arbitrary.

Finally, assume that $\{D,E,F\}\cap (P\cap M)=\emptyset$. Then none of the points of $P\cap M$ lies on the segments $\{(x,0),x\geq 2\}$ and $\{(0,y),y\geq 2\}$ and, as in the second case, we can prove that $E\in (P\cap M)$.
\end{proof}

\begin{Remark}
In higher dimension the analogue of Proposition \ref{2def} does not hold. Consider the polytope $P\subseteq\mathbb{Q}^3$ with vertexes 
$ (0, 0, 1), (1, 0, 2), (0, 2, 1), (2, 2, 1), (1, 1, 0)$.
The lattice points of $P$ are 
$$ (0, 0, 1), (1, 0, 2), (0, 2, 1), (2, 2, 1), 
(1, 1, 0), (1, 1, 1), (1, 2, 1), (0, 1, 1)$$
and hence the corresponding map to a projective space is given by
\stepcounter{thm}
\begin{equation}\label{map}
\begin{array}{ccc}
(\mathbb{C}^{*})^3 & \longrightarrow & \mathbb{P}^7\\
(x,y,z) & \mapsto & (xyz,x^2y^2z,z,xz^2,y^2z,xy,xy^2z,yz)
\end{array}
\end{equation}
Note that $P$ contains $(1,1,1)$ as an interior point, and hence it is not equivalent, modulo ${\rm GL}(3,\mathbb Z)$ and translations, to a polytope contained in the polytope of the degree two Veronese embedding of $\mathbb{P}^3$. Furthermore, $X_P$ is $2$-defective by Terracini's lemma. Now, the singular locus of $X_P$ is the union of seven invariant curves,
which correspond to the singular $2$-dimensional cones of the normal fan, and is stabilized by the action of the torus. Hence it corresponds via (\ref{map}) to the locus stabilized by the action of the torus on $\mathbb{C}^3$. Computing the differential of (\ref{map}) we get that the line $L$ corresponding to the plane $\{z=0\}\subseteq\mathbb{C}^3$ is in the singular locus of $X_P$. Hence, if $X_P$ is a cone, this line must be contained in its vertex. However, a line through a general point of $L$ and the point $(1,\dots,1)\in X_P$ is not entirely contained in $X_P$, and hence $X_P$ can not be a cone. The variety $X_P$ is a Gorenstein canonical toric Fano $3$-fold of degree $10$. Its entry in the \href{http://www.grdb.co.uk/forms/toricf3c}{Graded Ring Database} is 523456.
\end{Remark}

\section{Bounds for Segre-Veronese Varieties}\label{sec:boundsSV}
Let $SV^{n_1,\dots,n_r}_{d_1,\dots,d_r}$ be the Segre-Veronese variety  given as the image, in $\mathbb{P}^N$ with $N = \prod_{i=1}^r\binom{n_i+d_i}{d_i}-1$, of $\mathbb{P}^{n_1}\times\dots\times\mathbb{P}^{n_r}$ under the embedding induced by the complete linear system $\big|\mathcal{O}_{\mathbb{P}^{n_1}\times\dots\times\mathbb{P}^{n_r}}(d_1,\dots, d_r)\big|$. In the following we prove our main result. 

\begin{thm}\label{main_SV}
The Segre-Veronese variety $SV^{n_1,\dots,n_r}_{d_1,\dots,d_r}\subseteq\mathbb{P}^N$ is not $h$-defective for 
$$h\leq \frac{d_j}{n_j+d_j}\frac{1}{1+\sum_{i=1}^r n_i}\prod_{i=1}^r \binom{n_i+d_i}{d_i}$$
where $\frac{n_j}{d_j} = \max_{1\leq i\leq r}\left\lbrace\frac{n_i}{d_i}\right\rbrace$.
\end{thm}
\begin{proof}
Let $\Delta_{d_i}^{n_i}\subseteq\mathbb{Q}^{n_i+1}$ be the standard simplex. The polytope $P = \Delta_{d_1}^{n_1}\times\cdots\times \Delta_{d_r}^{n_r}$ has 
$$\prod_{i=1}^r\binom{d_i+n_i}{d_i}$$
integer points, and each facet is given by the Cartesian product of a facet of one of the $\Delta_{d_j}^{n_j}$ and the remaining $\Delta_{d_i}^{n_i}$ for $i\neq j$. Therefore, each facet contains
$$f_j=\binom{d_j+n_j-1}{d_j} \prod_{i\neq j}^r\binom{d_i+n_i}{d_i}$$ 
points for some $j$. Now, we compare the number of integer points on each facet:
\begin{align*}
f_j& \leq f_k\\
\binom{d_j+n_j-1}{d_j} \prod_{i\neq j}^r\binom{d_i+n_i}{d_i}  &  \leq 
\binom{d_k+n_k-1}{d_k} \prod_{i\neq k}^r\binom{d_i+n_i}{d_i}  \\
\binom{d_j+n_j-1}{d_j}\binom{d_k+n_k}{d_k} & \leq \binom{d_k+n_k-1}{d_k}  \binom{d_j+n_j}{d_j}\\
\dfrac{d_k+n_k}{n_k} & \leq 
\dfrac{d_j+n_j}{n_j} \\
\dfrac{d_k}{n_k} & \leq \dfrac{d_j}{n_j}
\end{align*}
Therefore the facet with maximum number of integer points is the one which minimizes $\frac{d_i}{n_i}$ that is maximizes $\frac{n_i}{d_i}$. Assume that $\frac{n_j}{d_j}=\max_{1\leq i\leq r}\left\lbrace\frac{n_i}{d_i}\right\rbrace$.

Since $P$ satisfies the conditions in Proposition \ref{hyp} the maximum number of integer points in a hyperplane section of $P$ is attained on a facet and in this case it is given by
$$\binom{d_j+n_j-1}{d_j} \prod_{i\neq j}^r\binom{d_i+n_i}{d_i}$$
Finally, to conclude it is enough to note that
\begin{align*}
&\frac{1}{1+\sum_in_i}
\left(
 \prod_{i=1}^r\binom{d_i+n_i}{d_i} 
 - 
 \binom{d_j+n_j-1}{d_j} \prod_{i\neq j}^r\binom{d_i+n_i}{d_i} \right)
 \\ & =  
 \frac{1}{1+\sum_in_i}
\binom{d_j+n_j-1}{d_j-1} \prod_{i\neq j}^r\binom{d_i+n_i}{d_i} 
 \\ & =   
 \frac{1}{1+\sum_in_i}
 \dfrac{d_j}{d_j+n_j}
 \prod_{i=1}^r\binom{d_i+n_i}{d_i} 
 \\ & =   
 \dfrac{1}{1+\dfrac{n_j}{d_j}}
 \frac{1}{1+\sum_in_i}
 \prod_{i=1}^r\binom{d_i+n_i}{d_i} 
\end{align*}
and to apply Theorem \ref{hyperprop}.
\end{proof}

\begin{Remark}
According to Theorem \ref{main_SV} we have a polynomial bound of degree $\sum_i n_i$ in the $d_i$, while
in the $n_i$ we have a polynomial bound of degree $\sum_id_i-2$.

A bound for non secant defectiveness of Segre varieties was given in \cite[Theorem 1.1]{Ge13} using the inductive machinery developed in \cite{AOP09}. When the numbers $n_i+1$ are powers of two \cite[Corollary 5.1]{Ge13} gives a sharp asymptotic bound for non secant defectiveness of Segre varieties. However, for general values of the $n_i$ the bound in \cite[Theorem 1.1]{Ge13} tends to zero when $r$ goes to infinity.
\end{Remark}

\begin{Proposition}\label{prop:twofactors}
The Segre-Veronese variety $SV^{1,n}_{2k+1,2}$ is not defective. Furthermore, $SV^{1,n}_{2k,2}$ is not $h$-defective for $h\leq k(n+1)$. 
\end{Proposition}
\begin{proof}
Let us begin with $SV^{1,n}_{2k+1,2}$. The corresponding polytope is $P=\Delta_{2k+1}^1 \times \Delta_2^{n}$ where
$$\Delta_{2k+1}^1=\{0,1,\dots,2k+1\}
 \mbox{ and }
\Delta_2^n=\{(x_1,\dots,x_n)\in \mathbb{Z}_{\geq 0}; \textstyle\sum x_j\leq 2\}$$
We view $P$ as a union of $2k+2$ floors labeled by $\Delta_{2k+1}^1$. We will triangulate each pair of floors. Note that it is enough to do this in the case $k=0$ where we have just two floors.

Consider the following disjoint subsets of $P$
$$
\begin{array}{cll}
S_1 & = & \{e_1+e_2\}\cup\{e_1+e_2+e_j;j=2\dots n+1\}\cup 
\{e_2+e_2\} \\ 
S_2 & =  & \{e_1+e_3\}\cup\{e_1+e_3+e_j;j=3\dots n+1\} \cup 
\{ e_3+e_j; j=2,3\} \\ 
\vdots &  &  \\ 
S_n & = & \{e_1+e_{n+1}\}\cup\{e_1+e_{n+1}+e_{n+1}\}\cup
\{e_{n+1}+e_j;j=2,\dots n+1\} \\ 
S_{n+1} & = & \{(0,\dots,0)\} \cup \{e_j; j=1\dots n+1\} 
\end{array} 
$$
Note that each set $S_i$ has cardinality $n+2$ and since $|P| = 2\binom{n+2}{2}=(n+1)(n+2)$ we have $P = \bigcup_{i=1}^{n+1}S_i$. Moreover each $S_i$ is an $(n+1)$-simplex in $\mathbb{Q}^{n+1}$.
 
Now, consider integers 
$$b_1 \gg b_2 \gg\dots \gg b_{n+1}>0$$
and vectors
$$
\begin{array}{cll}
v_1 & = & (b_1,b_2,0,\dots,0) \\ 
v_2 & =  & (b_1,b_3,b_2,0,\dots,0)\\ 
v_3 & = & (b_1,b_4,b_3,b_2,0,\dots,0)\\
\vdots & &  \\ 
v_n & = & (b_1,b_{n+1},\dots,b_2)\\ 
v_{n+1} & = & (1,1,\dots,1) 
\end{array} 
$$

We will show that these vectors and simplexes make Theorem~\ref{teo:main} work. In the first step in order to maximize $\left\langle b(\Delta),v_1\right\rangle$ we need that $\Delta$ has the maximum possible number of points on the top floor, corresponding to $e_1$. Furthermore, since $e_2$ appears in all the vectors of $S_1$ and $b_2\gg b_3\dots \gg b_{n+1}$ among the simplexes having $n+1$ points on the top floor the one maximizing $\langle b(\Delta),v_1\rangle$ is $S_1$. Therefore, $v_1$ separates $S_1$. 

Now, note that the remaining points on the top floor are exactly the ones in the hyperplane $x_2=0$. Then, among the simplexes with points in $S\setminus S_1$ the ones maximizing $\langle b(\Delta), v_2\rangle$ must have $n$ points on the top floor and two on the bottom floor. Since $b_2\gg b_3$ the points on the top floor must have the third coordinate non zero, and since there are exactly $n$ of these we have to take all of them. By the same argument on the bottom floor we have to take $(0,1,1,0,\dots,0)$ and $(0,0,2,0,\dots,0)$. Hence, $v_2$ separates $S_2$.

Now, the remaining points on the top floor are in the linear space $x_2=x_3=0$. Arguing similarly we see that $v_1,\dots,v_n$ separate $S_1,\dots, S_n$. In the last step there are just $n+2$ points left and these form a simplex. Setting $S_{n+1}=\Delta \setminus \cup_{i=1}^n S_i$ any vector $v_{n+1}$ will do.

Therefore, for each pair of floors we construct $n+1$ simplexes and since we have $k+1$ pairs of floors Theorem~\ref{teo:main} yields that $SV^{1,n}_{2k+1,2}\subseteq\mathbb{P}^N$ is not $h$-defective for $h\leq (k+1)(n+1)$. Then 
$$\dim\Sec_{(k+1)(n+1)}(SV^{1,n}_{2k+1,2}) = (k+1)(n+1)^2+(k+1)(n+1)-1 = (k+1)(n+1)(n+2)-1 = N$$
and $SV^{1,n}_{2k+1,2}\subseteq\mathbb{P}^N$ is not defective. 

Now, consider $SV^{1,n}_{2k,2}$. In this case we have $2k+1$ floors. Considering just the first $2k$ of them and arguing as in the previous case we get that $SV^{1,n}_{2k,2}$ is not $h$-defective for $h\leq k(n+1)$. 
\end{proof}

\begin{Remark}
The non secant defectiveness of $SV^{1,n}_{2k+1,2}$ was proven, by different methods, in \cite[
Proposition 3.1]{Ab08}. Furthermore, by \cite[Proposition 3.2]{Ab08} $SV^{1,n}_{2k,2}$ is $h$-defective for $k(n+1)+1\leq h\leq k(n+1)+n$.
\end{Remark}

\subsection{Identifiability}
Let $X\subseteq\mathbb{P}^N$ be an irreducible non-degenerated variety. A point $p\in\mathbb{P}^N$ is said to be $h$-identifiable, with respect to $X$, if it lies on a unique $(h-1)$-plane $h$-secant to $X$. We say that $X$ is $h$-identifiable if the general point of $\Sec_h(X)$ is $h$-identifiable.

\begin{Corollary}\label{Id_Toric}
Let $P\subseteq M_{\mathbb{Q}}$ be a full-dimensional lattice polytope, $X_P$ the corresponding $n$-dimensional toric variety, and $m$ the maximum number of points on a hyperplane section of $P\cap M$. Assume that $2n < \frac{|P\cap M|-m}{n+1}$. Then, for
$$h\leq\dfrac{|P\cap M|-m}{n+1}$$
$X_P$ is $(h-1)$-identifiable. 
\end{Corollary}
\begin{proof}
It is enough to apply Theorem \ref{hyperprop} and \cite[Theorem 3]{CM19}.
\end{proof}

\begin{Corollary}\label{Id_SV}
Consider the Segre-Veronese variety $SV^{n_1,\dots,n_r}_{d_1,\dots,d_r}\subseteq\mathbb{P}^N$. Set $\frac{n_j}{d_j} = \max_{1\leq i\leq r}\left\lbrace\frac{n_i}{d_i}\right\rbrace$, and assume that $2\sum_{i=1}^r n_i < \frac{d_j}{n_j+d_j}\frac{1}{1+\sum_{i=1}^r n_i}\prod_{i=1}^r \binom{n_i+d_i}{d_i}$. Then, for  
$$h\leq \frac{d_j}{n_j+d_j}\frac{1}{1+\sum_{i=1}^r n_i}\prod_{i=1}^r \binom{n_i+d_i}{d_i}$$
$SV^{n_1,\dots,n_r}_{d_1,\dots,d_r}\subseteq\mathbb{P}^N$ is $(h-1)$-identifiable. 
\end{Corollary}
\begin{proof}
It is enough to apply Theorem \ref{main_SV} and \cite[Theorem 3]{CM19}.
\end{proof}

\section{A Magma script}\label{Magma}
A Magma library which implements the
following algorithm can be downloaded
at the following link:

\begin{center}
\url{https://github.com/alaface/secant-algorithm}
\end{center}

In this section we present an algorithmic implementation of 
Theorem \ref{teo:main} and show what kind of results it can 
provide for Segre-Veronese varieties. 
In what folows $M\simeq\mathbb Z^n$ and
$N := {\rm Hom}(M,\mathbb Z)$ is its dual.
Denote by $M_{\mathbb Q}$ the corresponding 
rational vector space. Giving a subset
$S\subseteq M_{\mathbb Q}$ we say that
$S$ is independent if it is affinely 
independent and we say that it is 
full-dimensional if its affine span is the 
whole space.
Our main algorithm is the following.\\

\begin{tcolorbox}
\begin{algorithm}[H]
 \SetKwInOut{Input}{Input}
 \SetKwInOut{Output}{Output}
 \Input{a finite, full-dimensional subset $S\subseteq M$}
 \While{$S$ is full-dimensional}{
choose $v\in N_{\mathbb Q}$ such that 
$\varphi_v$ is injective on $S$\; 
reorder $S$ increasingly according 
to $\varphi_v$\;
define $\Delta := \{\max(S)\}$\;
 \Repeat{$\Delta$ is full-dimensional}{
 $x := \max\{u\in S\setminus \Delta\, 
:\, \Delta\cup\{ u\}\text{ is independent}\}$\;
$\Delta := \Delta\cup\{x\}$\;
 }
 $S := S\setminus\Delta$;
 }
  \lIf{$S$ is independent}{
   \Return false
   } \lElse {
   \Return true
  }
\end{algorithm}
\end{tcolorbox}

This procedure can show that a toric variety is not defective but can not determine whether it is defective. Furthermore, some details must be considered.
That is if the output is {\tt false} then 
all the secant varieties of the toric variety $X_S$ 
are not defective. On the other hand there is no 
guarantee that if the output is {\tt true} then $X_S$
admits a defective $r$-secant variety for some $r$.
Due to this we sometimes apply the above algorithm
several times to improve the possibility of getting
a correct result in case the output is {\tt true}.
%First, it can happen that several times the randomly chosen vector $v$ does not separate a simplex in $S,$ and the computer takes to much time to find a vector that works.  In this case the script would run steps $(2)$ to $(4)$ for too long. In order to avoid this one can fix a maximum number of tries. 
%Another issue is that if $X_P$ is defective the script will run forever going from step $(1)$ to step $(6)$ and then back to step $(1)$. Again one can fix a maximum number of tries to avoid this.
%Unfortunately in both cases we can not distinguish whether the variety is defective or not, and indeed is impossible for the computer to find a good vector $v$ or a good triangulation, or there is a solution but the computer just did not find it yet. Despite these limitations 
We where able to use an implementation of this algorithm in MAGMA \cite{Magma97} in order to find several non defective Segre-Veronese varieties.

First we look at Segre-Veronese varieties with two factors $SV_{(d_1,d_2)}^{(n_1,n_2)}$. We assume that $n_1\leq n_2$ and $n_2>1$ since by \cite[Theorem 2.2]{LP13} $SV_{(d_1,d_2)}^{(1,1)}$ is defective if and only if $d_1=2$ and $d_2$ is even. We also assume that $(d_1,d_2)\neq (1,1)$ since Segre varieties with two factors are almost always defective.

If either $n_1=1$ or $n_1=2$ we get the results listed in Tables \ref{tab:P1Pn} and \ref{tab:P2Pn}. The only cases where the script was unable to prove the non defectiveness are the already known ones \cite[Conjecture 5.5 (b)(d)]{AB13} and
\cite[Conjecture 5.5 (a)(c)(e)]{AB13} respectively.
\begin{table}[h!]
\begin{tabular}{c|c|c|c}
$(n_1,n_2)$ & $(d_1,d_2)\neq (1,1)$ & known defective cases & possible new defective cases\\
\hline
$(1,2)$     &  $d_1+d_2\leq 40$         &                $(1,3),(2k,2), 1\leq k \leq 19$ &                                   none \\ \hline
$(1,3)$     &   $d_1+d_2\leq 20$            &                $(2k,2), 1\leq k \leq 9$  &                                   none \\ \hline
$(1,4)$     &   $d_1+d_2\leq 10$           &                $(2k,2), 1\leq k \leq 4$ &                                  none \\ \hline
$(1,5)$     &   $d_1+d_2\leq 9$          &                $(2k,2), 1\leq k \leq 3$&                                              none\\ \hline
$(1,6)$     &   $d_1+d_2\leq 5$           &                $(2,2)$ & none \\ \hline
$(1,7)$     &   $d_1+d_2\leq 3$           &                none &     none                              
\end{tabular}\caption{\footnotesize{Script results for $SV_{(d_1,d_2)}^{(1,n_2)}$}}\label{tab:P1Pn}
\end{table}

\begin{table}[h!]
\begin{tabular}{c|c|c|c}
$(n_1,n_2)$ & $(d_1,d_2)\neq (1,1)$ & known defective cases & possible new defective cases\\
\hline
$(2,2)$     &  $d_1+d_2\leq 23$         &                $(2,2)$ &                                   none \\ \hline
$(2,3)$     &   $d_1+d_2\leq 10$            &                $(1,2),(2,2)$  &                                   none\\ \hline
$(2,4)$     &   $d_1+d_2\leq 6$           &                $(2,2)$ &                                  none \\ \hline
$(2,5)$     &   $d_1+d_2\leq 4$          &                $(1,2),(2,1),(2,2)$&                                              none \\ \hline
$(2,6)$     &   $d_1+d_2\leq 3$           &                $(2,1)$ & none                         
\end{tabular}\caption{\footnotesize{Script results for $SV_{(d_1,d_2)}^{(2,n_2)}$}}\label{tab:P2Pn}
\end{table}

For $3\leq n_1\leq 4, n_1\leq n_2 \leq 5$ we found six cases, listed in Table \ref{tab:PnPm}, where the computer was unable to check whether the corresponding Segre-Veronese variety is defective or not. Again these cases already appeared in the literature \cite[Conjecture 5.5 (c)(e)]{AB13}.

\begin{table}[h!]
\begin{tabular}{c|c|c|c}
$(n_1,n_2)$ & $(d_1,d_2)\neq (1,1)$ &known defective cases & possible new defective cases\\
\hline
$(3,3)$     &  $d_1+d_2\leq 8$         &                $(2,2)$ &                                   none \\ \hline
$(3,4)$     &   $d_1+d_2\leq 5$            &                $(2,1),(2,2)$  &                                   none \\ \hline
$(3,5)$     &   $d_1+d_2\leq 4$           &                $(2,2),(3,1)$ &                                  none \\ \hline
$(4,4)$     &   $d_1+d_2\leq 5$          &                $(2,2)$&                                              none \\ \hline
$(4,5)$     &   $d_1+d_2\leq 3$           &                none & none
\end{tabular}\caption{\footnotesize{Script results for $SV_{(d_1,d_2)}^{(n_1,n_2)}, 3\leq n_1\leq 4, n_1\leq n_2 \leq 5$}}\label{tab:PnPm}
\end{table}

Now, we proceed with Segre-Veronese varieties with three factors $SV_{(d_1,d_2,d_3)}^{(n_1,n_2,n_3)}$. We assume that $n_1\leq n_2\leq n_3$ and $n_3>1$ since \cite[Theorem 2.2]{LP13} classifies defective products of $\mathbb{P}^1$. If $n_1=n_2$ we assume that $d_1\leq d_2$ and similarly for $n_2=n_3$ we assume that $d_2\leq d_3$. By \cite{CGG05} the following Segre-Veronese varieties are defective:$$ 
SV_{(1,1,2)}^{(1,1,2)}, SV_{(1,1,2)}^{(1,1,3)}, SV_{(1,1,2)}^{(1,1,4)}, SV_{(1,1,2)}^{(1,1,5)}, SV_{(1,1,2)}^{(1,1,6)}, SV_{(2,2,2)}^{(1,1,2)}, SV_{(2,2,2)}^{(1,1,3)},
SV_{(1,3,1)}^{(1,1,2)}, SV_{(1,4,1)}^{(1,1,3)},
SV_{(1,5,1)}^{(1,1,4)}$$
$$ 
SV_{(2k,1,1)}^{(1,2,2)}, SV_{(5,1,1)}^{(1,2,3)},
SV_{(6,1,1)}^{(1,2,4)}, SV_{(2k,1,1)}^{(1,3,3)},
SV_{(2,1,1)}^{(2,2,2)}, SV_{(2,1,1)}^{(2,3,3)}
$$

The following ones are also defective  by \cite[Theorem 2.4]{CGG08} since they are unbalanced:
$$
SV_{(1,1,1)}^{(1,1,3)}, SV_{(1,1,1)}^{(1,1,4)}, SV_{(1,1,1)}^{(1,1,5)}, SV_{(1,2,1)}^{(1,1,5)}, SV_{(1,1,1)}^{(1,2,4)}, SV_{(1,1,1)}^{(1,1,6)}, SV_{(1,2,1)}^{(1,1,6)}, SV_{(1,1,1)}^{(1,2,5)}$$

The variety $SV_{(1,1,1)}^{(2,2,2)}$ is defective by \cite[Theorem 3.1]{LM08} and $SV_{(1,1,1)}^{(2,3,3)}$ is defective by \cite[Proposition 4.10]{AOP09}. 
In Tables \ref{tab:P1P1Pn}, \ref{tab:P1PnPm} and \ref{tab:P2PnPm} we present the results found for Segre-Veronese of three factors. We were unable to check, using our script, whether the following Segre-Veronese varieties 
$$SV_{(1,5,1)}^{(1,1,2)},
SV_{(1,8,1)}^{(1,1,2)},SV_{(1,10,1)}^{(1,1,2)},
SV_{(1,3,1)}^{(1,1,3)},SV_{(1,6,1)}^{(1,1,3)},
SV_{(1,7,1)}^{(1,1,3)},SV_{(1,9,1)}^{(1,1,3)},
SV_{(1,4,1)}^{(1,1,4)},SV_{(1,7,1)}^{(1,1,4)}$$
$$ 
SV_{(1,4,1)}^{(1,1,5)},
SV_{(1,5,1)}^{(1,1,5)},SV_{(2,1,1)}^{(1,2,3)},SV_{(3,1,1)}^{(1,2,3)},
SV_{(7,1,1)}^{(1,2,3)},SV_{(3,1,1)}^{(1,2,4)},
SV_{(5,1,1)}^{(1,2,4)},
SV_{(2,1,1)}^{(1,3,4)}.$$
are defective or not.

\begin{table}[h!]
\begin{tabular}{c|c|c|c|c}
$(n_1,n_2,n_3)$ & $(d_1,d_2,d_3), d_1\leq d_2$ & known defective cases & \makecell{ possible new\\ defective cases} & 
\makecell{ new\\ defective cases} \\
\hline
$(1,1,2)$     &  $d_1+d_2+d_3\leq 13$         &  
 $(1,1,2),(1,3,1),(2,2,2)$ & none &                                   \makecell{ $(1,5,1),(1,8,1)$\\ $(1,10,1)$} \\ \hline
  $(1,1,3)$     &  $d_1+d_2+d_3\leq 11$         & \makecell{ $(1,1,1),(1,1,2)$ \\ $(1,4,1),(2,2,2)$ }& 
 none & \makecell{ $  (1,3,1),(1,6,1)$ \\ $(1,7,1),(1,9,1)$                               } \\ \hline
$(1,1,4)$     &  $d_1+d_2+d_3\leq 9$         &  $(1,1,1),(1,1,2),(1,5,1)$   &                                   none & $(1,4,1), (1,7,1)$\\ \hline
$(1,1,5)$     &  $d_1+d_2+d_3\leq 7$         &     $(1,1,1),(1,1,2),(1,2,1)$  & none   & $ (1,4,1),(1,5,1) $ \\ \hline
$(1,1,6)$     &  $d_1+d_2+d_3\leq 4$         &    $(1,1,1),(1,1,2),(1,2,1)$ &                                  none & none
\end{tabular}\caption{\footnotesize{Script results for $SV_{(d_1,d_2,d_3)}^{(1,1,n_3)}$}}\label{tab:P1P1Pn}
\end{table}

The defectiveness of the cases in the last column of Table \ref{tab:P1P1Pn} is proved in Propositions \ref{Giorgio}, \ref{prop_t4} and Corollary \ref{cor:def3factors}.

\begin{table}[h!]
\begin{tabular}{c|c|c|c}
$(n_1,n_2,n_3)$ & $(d_1,d_2,d_3)$ & known defective cases & possible new defective cases \\
\hline
$(1,2,2)$     &  $d_1+d_2+d_3\leq 11$         &   $(2,1,1),(4,1,1),(6,1,1),(8,1,1)$ &                                   none \\ \hline
  $(1,2,3)$     &  $d_1+d_2+d_3\leq 9$         &  $(5,1,1)$   &                                   $(2,1,1),(3,1,1),(7,1,1)$\\ \hline
$(1,2,4)$     &  $d_1+d_2+d_3\leq 7$         &   $(1,1,1)$  &                                   $(3,1,1),(5,1,1)$ \\ \hline
$(1,2,5)$     &  $d_1+d_2+d_3\leq 4$         & 
$(1,1,1)$  &                                   none\\ \hline
  $(1,3,3)$     &  $d_1+d_2+d_3\leq 7$         &   $(2,1,1),(4,1,1)$  &                                   none\\ \hline
  $(1,3,4)$     &  $d_1+d_2+d_3\leq 4$         &  none &            $(2,1,1)$                        
\end{tabular}\caption{\footnotesize{Script results for $SV_{(d_1,d_2,d_3)}^{(1,2,n_3)}$ and
$SV_{(d_1,d_2,d_3)}^{(1,3,n_3)}$}}\label{tab:P1PnPm}
\end{table}

We did not manage to prove that the cases in the last column of Table \ref{tab:P1PnPm} are indeed defective. 

\begin{table}[h!]
\begin{tabular}{c|c|c|c}
$(n_1,n_2,n_3)$ & $(d_1,d_2,d_3)$ & known defective defective cases & possible new defective cases \\
\hline
$(2,2,2)$     &  $d_1+d_2+d_3\leq 9$         &  \makecell{   $
(1,1,1),(1,1,2)$ } &                                   none \\ \hline
  $(2,2,3)$     &  $d_1+d_2+d_3\leq 6$         &  none &                                   none
   \\ \hline
$(2,2,4)$     &  $d_1+d_2+d_3\leq 4$         &   
none  &                                  none \\ \hline
  $(2,3,3)$     &  $d_1+d_2+d_3\leq 4$         &$(1,1,1),(2,1,1)$  & none                                   
\end{tabular}\caption{\footnotesize{Script results for $SV_{(d_1,d_2,d_3)}^{(2,n_2,n_3)}$}}\label{tab:P2PnPm}
\end{table}

The following Magma script shows how
to check the results listed in the above tables.
In the specific case we are listing the
defective Segre-Veronese varieties with
$[n_1,n_2] = [1,2]$ and $1\leq d_1,d_2\leq 10$.

\begin{tcolorbox}
{\footnotesize
\begin{verbatim}
> load "library.m";
> dd := [[d1,d2] : d1,d2 in [1..10]];
> for d in dd do
 if IsSVDef([1,2],d,5) then d; end if;
 end for;

[ 1, 3 ]
[ 2, 2 ]
[ 4, 2 ]
[ 6, 2 ]
[ 8, 1 ]
[ 8, 2 ]
[ 10, 2 ]
\end{verbatim}
}
\end{tcolorbox}

Observe that case $[d_1,d_2] = [8,1]$ has been
recognized by the program as a defective one.
Anyway if one runs the function enough times
{\tt IsSVDef([1,2],[8,1],5)} then at some point 
the output will be {\tt false}.

Our second Magma example compares the 
running times for checking non-speciality 
of the Segre-Veronese varieties $[n_1,n_2]
= [1,2]$ embedded with multidegrees 
$[d_1,d_2] = [13,13]$ and $[n_1,n_2,n_3]
= [2,2,2]$ embedded with multidegrees 
$[d_1,d_2,d_3] = [2,2,6]$.
The first function
{\tt IsSVDef} is based on our algorithm.
The second function makes use of 
the classical Terracini's lemma which 
reduces the defectivity checking to the 
calculation of the dimension of a linear 
system of affine hypersurfaces through 
double points in general position.

\begin{tcolorbox}
{\footnotesize
\begin{verbatim}
> load "library.m";
> time IsSVDef([1,2],[13,13],5);
false
Time: 4.480
> time IsSpecial(ProjSpaces([1,2],[13,13]));
false
Time: 198.190

> time IsSVDef([2,2,2],[2,2,6],10);           
false
Time: 3.510
> time IsSpecial(ProjSpaces([2,2,2],[2,2,6]));
false
Time: 30.110
\end{verbatim}
}
\end{tcolorbox}

According to our tests we found that the
difference between the computational times
of the above two functions increases according
to the number of points of the Riemann-Roch 
polytope of the toric variety.

\section{New examples of defective Segre-Veronese varieties}\label{sec:def}
In this last section we give examples of defective Segre-Veronese varieties using three different methods. Namely, by the general theory of flattenigs in Section \ref{flat}, by constructing low degree rational normal curves in Segre-Veronese varieties in Section \ref{rnc}, and by producing special Cremona transformations of product of projective lines in Section \ref{Cremona}. As noticed in Remarks \ref{rem1} and \ref{rem2} the defective Segre-Veronese varieties in Sections \ref{rnc} and \ref{Cremona} were already well know even though the methods we present are new.  

\stepcounter{thm}
\subsection{Flattenings}\label{flat}
Let $V_1,...,V_{p}$ be vector spaces of finite dimension, and consider the tensor product $V_1\otimes ...\otimes V_{p} = (V_{a_1}\otimes ...\otimes V_{a_s})\otimes (V_{b_1}\otimes ...\otimes V_{b_{p-s}})= V_{A}\otimes V_{B}$ with $A\cup B = \{1,...,p\}$, $B = A^c$. Then we may interpret a tensor 
$$T \in V_1\otimes ...\otimes V_p = V_{A}\otimes V_{B}$$
as a linear map $\widetilde{T}:V_{A}^{*}\rightarrow V_{A^c}$. Clearly, if the rank of $T$ is at most $r$ then the rank of $\widetilde{T}$ is at most $r$ as well. Indeed, a decomposition of $T$ as a linear combination of $r$ rank one tensors yields a linear subspace of $V_{A^c}$, generated by the corresponding rank one tensors, containing $\widetilde{T}(V_A^{*})\subseteq V_{A^c}$. The matrix associated to the linear map $\widetilde{T}$ is called an \textit{$(A,B)$-flattening} of $T$.    

In the case of mixed tensors we can consider the embedding
$$\Sym^{d_1}V_1\otimes ...\otimes \Sym^{d_p}V_p\hookrightarrow V_A\otimes V_B$$
where $V_A = \Sym^{a_1}V_1\otimes ...\otimes \Sym^{a_p}V_p$,
$V_B=\Sym^{b_1}V_1\otimes ...\otimes\Sym^{b_p}V_p$, with $d_i =
a_i+b_i$ for any $i = 1,...,p$. In particular, if $n = 1$ we may
interpret a tensor $F\in \Sym^{d_1}V_1$ as a degree $d_1$ homogeneous
polynomial on $\mathbb{P}(V_1^*)$. In this case the matrix associated
to the linear map $\widetilde{F}:V_A^*\rightarrow V_B$ is nothing but
the $a_1$-th \textit{catalecticant matrix} of $F$, that is the matrix
whose rows are the coefficient of the partial derivatives of order
$a_1$ of $F$.

\begin{Remark}\label{flat_cat}
Consider a tensor $T\in \Sym^{d_1}\mathbb{C}^{n_1+1}\otimes \Sym^{d_2}\mathbb{C}^{n_2+1}\otimes \Sym^{d_3}\mathbb{C}^{n_3+1}$ and the flattening
$$\Sym^{d_1}\mathbb{C}^{n_1+1}\otimes\Sym^{d_2-k}\mathbb{C}^{n_2+1}\rightarrow \Sym^{k}\mathbb{C}^{n_2+1}\otimes\Sym^{d_3}\mathbb{C}^{n_3+1}$$
Fix coordinates $x_0,\dots,x_{n_{2}}$ on $\mathbb{C}^{n_2+1}$ and $v_0,\dots,v_{n_3}$  on $\mathbb{C}^{n_3+1}$. Then the matrix of the above flattening has the following form
$$
\left(\begin{array}{c}
\frac{\partial^{d_3}}{\partial v_0^{d_3}}\frac{\partial^k}{\partial x_0^{k}}T\\ 
\vdots\\ 
\frac{\partial^{d_3}}{\partial v_0^{d_3}}\frac{\partial^k}{\partial x_{n_2}^{k}}T\\ 
\vdots\\ 
\frac{\partial^{d_3}}{\partial v_{n_3}^{d_3}}\frac{\partial^k}{\partial x_0^{k}}T\\ 
\vdots\\ 
\frac{\partial^{d_3}}{\partial v_{n_3}^{d_3}}\frac{\partial^k}{\partial x_{n_2}^{k}}T
\end{array}\right) 
$$
Note that since $T$ has $\binom{n_2+k}{n_2}$ partial derivatives of order $k$ with respect to $x_0,\dots,x_{n_2}$ and each of these derivatives has in turn $\binom{n_3+d_3}{n_3}$ partial derivatives of order $d_3$ with respect to $v_0,\dots,v_{n_3}$ this is a matrix of size $\binom{n_3+d_3}{n_3}\binom{n_2+k}{n_2}\times \binom{n_2+d_2-k}{n_2}\binom{n_1+d_1}{n_1}$.

In general, the $(h+1)\times (h+1)$ minors of the above matrix yield equations for $\Sec_h(SV^{(n_1,n_2,n_3)}_{(d_1,d_2,d_3)})$. However, in practice it is hard to compute the codimension of the variety cut out by these minors. In a Magma script, that can be found as an ancillary file in the arXiv version of the paper, we manage to simplify the computations. The script reduces the equations given by the minors to positive characteristic. The variety cut out by these reduced equations has dimension greater or equal than our original variety. So if this dimension is strictly less that the expected dimension of $\Sec_h(SV^{(n_1,n_2,n_3)}_{(d_1,d_2,d_3)})$ we get that $SV^{(n_1,n_2,n_3)}_{(d_1,d_2,d_3)}$ is $h$-defective. 
\end{Remark}

\begin{Proposition}\label{Giorgio}
The Segre-Veronese variety $SV^{(1,1,2)}_{(1,5a+3,1)}$ is $(6a+5)$-defective for all $a\geq 0$, and the Segre-Veronese variety $SV^{(1,1,2)}_{(1,5a+5,1)}$ is $(6a+7)$-defective for all $a\geq 0$. 
\end{Proposition}
\begin{proof}
Let us begin with $SV^{(1,1,2)}_{(1,5a+3,1)}\subset\mathbb{P}^{30a+23}$. The $(6a+5)$-secant variety of $SV^{(1,1,2)}_{(1,5a+3,1)}$ is expected to fill the ambient projective space. On the other hand, we may consider the following flattening:
$$\mathbb{C}^2\otimes\Sym^{3a+2}\mathbb{C}^2\rightarrow \Sym^{2a+1}\mathbb{C}^2\otimes\mathbb{C}^3$$
By Remark \ref{flat_cat} the matrix associated to this linear map is a $(6a+6)\times (6a+6)$ block matrix where the blocks are catalecticant matrices. So the determinant of this matrix yields a non trivial equation for $\Sec_{6a+5}(SV^{(1,1,2)}_{(1,5a+3,1)})\subset\mathbb{P}^{30a+23}$.

Now, consider $SV^{(1,1,2)}_{(1,5a+5,1)}\subset\mathbb{P}^{30a+35}$. In this case $\Sec_{6a+7}(SV^{(1,1,2)}_{(1,5a+3,1)})$ is expected to be a hypersurface in $\mathbb{P}^{30a+35}$. Consider the following flattening:
$$\mathbb{C}^2\otimes\Sym^{3a+3}\mathbb{C}^2\rightarrow \Sym^{2a+2}\mathbb{C}^2\otimes\mathbb{C}^3$$
Note that the source and the target vector spaces have dimension $6a+8$ and $6a+9$ respectively. By Remark \ref{flat_cat} taking the minors of size $6a+8$ of the corresponding $(6a+9)\times (6a+8)$ matrix we get at least two independent equations for $\Sec_{6a+7}(SV^{(1,1,2)}_{(1,5a+5,1)})\subset\mathbb{P}^{30a+35}$.
\end{proof}

\begin{Proposition}\label{Giorgiobis}
Let $n,d\geq 2$ and assume that there exist $d_1,d_2\geq 1$ such that $2(d_1+1)=(d_2+1)(n+1)$. Then $SV^{(1,1,n)}_{(1,d,1)}$ is $(2d_1+1)$-defective.
\end{Proposition}
\begin{proof} Proceeding as in the first part of the proof of Proposition \ref{Giorgio} we consider the flattening
$$\mathbb{C}^2\otimes\Sym^{d_1}\mathbb{C}^2\rightarrow \Sym^{d_2}\mathbb{C}^2\otimes\mathbb{C}^n$$
Note that $\Sec_{(2(d_1+1)-1)}(SV^{(1,1,n)}_{(1,d,1)})$ is expected to fill the ambient projective space. However the above flattening yields at least one non trivial equation for $\Sec_{(2(d_1+1)-1)}(SV^{(1,1,n)}_{(1,d,1)})$.
\end{proof}

\begin{Corollary}\label{cor:def3factors}
The Segre-Veronese variety $SV^{(1,1,n)}_{(1,a(n+3)-2,1)}$ is $(2a(n+1)-1)$-defective for all $a\geq 0.$ Moreover, if $n$ is odd then the Segre-Veronese variety $SV^{(1,1,n)}_{(1,a(n+3)/2-2,1)}$ is $(a(n+1)-1)$-defective for all $a\geq 0$. In particular $SV^{(1,1,3)}_{(1,7,1)}$ is $11$-defective.
\end{Corollary}
\begin{proof}
Take $d_1=a(n+1)-1$ and $d_1=\frac{a(n+1)}{2}-1$ in Proposition \ref{Giorgiobis}.
\end{proof}

\begin{Proposition}\label{prop_t4}
For $n = 3$ and $d\in\{3,6,9\}$, $n = 4$ and $d\in\{4,7\}$, $n = 5$ and $d\in\{4,5\}$ the Segre-Veronese variety $SV^{(1,1,n)}_{(1,d,1)}$ is $h$-defective with $h = 5, h = 9, h = 13, h = 7, h = 11, h = 7$ and $h = 9$ respectively.
\end{Proposition}
\begin{proof}
This is an application of the Magma script described in the last part of Remark \ref{flat_cat}.
\end{proof}

\stepcounter{thm}
\subsection{Rational normal curves and defectiveness}\label{rnc}
In some particular cases defectiveness can be proved by producing low degree rational curves through a certain number of general points on a Segre-Veronese variety. 

\begin{Lemma}\label{lemma_RC}
Consider the product $X = \mathbb{P}^{n_1}\times\dots\times\mathbb{P}^{n_r}$ with $n_1 < n_2\leq \dots \leq n_r$. There exists a rational curve in $X$ of multi-degree $(n_1,\dots,n_r)$ through $n_1+3$ general points $p_1,\dots,p_{n_1+3}\in X$.
\end{Lemma}
\begin{proof}
Let us begin with the case $n_2 = \dots = n_r = n_1+1$. We view $\mathbb{P}^{n_1}$ as a linear subspace of $\mathbb{P}^{n_2}\subseteq\dots\subseteq\mathbb{P}^{n_r}$, and write $p_i = (p_i^1,\dots,p_i^r)$ where $p_i^j\in\mathbb{P}^{n_j}$. Without loss of generality we may assume that $p_1^1,\dots,p_{n_1+2}^1\in\mathbb{P}^{n_1}$ are the projections from $p_{n_1+3}^j$ of $p_1^j,\dots,p_{n_1+2}^j$ for all $j = 2,\dots,r$. 

Let $C_1\subseteq\mathbb{P}^{n_1}$ be the unique rational normal curve of degree $n_1$ through $p_1^1,\dots,p_{n_1+3}^1$. This is the image of a morphism $\gamma_1:\mathbb{P}^1\rightarrow C_1\subseteq\mathbb{P}^{n_1}$ of degree $n_1$ such that $\gamma_1(x_k) = p_{k}^1$ for $k = 1,\dots,n_1+3$ where $x_1,\dots,x_{n_1+3}\in\mathbb{P}^1$.

Now, consider a projective space $\mathbb{P}^{n_i}$ with $i > 1$. The rational normal curves in $\mathbb{P}^{n_i}$ through $p_1^i,\dots,p_{n_1+2}^i$ form a family of dimension greater than or equal to $n_1-1$ and the equality holds if and only if $n_i = n_1+1$. Among these curves there is one $\gamma_i:\mathbb{P}^1\rightarrow C_i\subseteq\mathbb{P}^{n_i}$ whose tangent direction at $p_{n_i+3}^j$ is given by the line $\left\langle p_{n_1+3}^j, p_{n_1+3}^1\right\rangle$ and such curve is unique if and only if $n_i = n_1+1$. Hence, we have the following commutative diagram
\[
  \begin{tikzpicture}[xscale=3.5,yscale=-1.2]
    \node (A0_0) at (0, 0) {$\mathbb{P}^1$};
    \node (A0_1) at (1, 0) {$C_i$};
    \node (A1_1) at (1, 1) {$C_1$};
    \path (A0_0) edge [->]node [auto] {$\scriptstyle{\gamma_i}$} (A0_1);
    \path (A0_1) edge [->]node [auto] {$\scriptstyle{\pi_i}$} (A1_1);
    \path (A0_0) edge [->]node [auto,swap] {$\scriptstyle{\gamma_1}$} (A1_1);
  \end{tikzpicture}
  \] 
where $\pi_i:C_i\rightarrow C_1$ is the morphism induced by the projection from $p_{n_1+3}^i$. Consider the points $y_j = \gamma_i^{-1}(p_j^i)$ for $j = 1,\dots,n_1+3$. The automorphism $\gamma_1^{-1}\circ\pi_i\circ\gamma_i\in PGL(2)$ maps $y_j$ to $x_j$, and we may use it to reparametrize $\gamma_i$ to a curve $\pi_i^{-1}\circ \gamma_1:\mathbb{P}^1\rightarrow C_i\subseteq\mathbb{P}^{n_i}$ such that $(\pi_i^{-1}\circ \gamma_1)(x_j) = p_{j}^i$ for $j = 1,\dots,n_1+3$. 

Finally, the map 
$$
\begin{array}{cccc}
\gamma: &\mathbb{P}^1& \longrightarrow & C\subseteq \mathbb{P}^{n_1}\times\dots\times\mathbb{P}^{n_r}\\
      & t & \longmapsto & (\gamma_1(t),\dots,\gamma_r(t))
\end{array}
$$
yields a curve of multi-degree $(n_1,\dots,n_r)$ in $\mathbb{P}^{n_1}\times\dots\times\mathbb{P}^{n_r}$ such that $\gamma(x_i) = p_i = (p_i^1,\dots,p_i^r)$ for $i = 1,\dots,n_1+3$.

When $n_i> n_1+1$ first we project $C_i$ from a certain number of general points in order to reach a projective space of dimension $n_1+1$ and then we apply the argument above.
\end{proof}

\begin{Proposition}\label{def_VSP}
The Segre-Veronese varieties $SV_{(1,1,1)}^{(2,2,2)}$ and $SV_{(1,1,1)}^{(2,3,3)}$ are respectively $4$-defective and $5$-defective.
\end{Proposition}
\begin{proof}
Let us begin with $SV_{(1,1,1)}^{(2,3,3)}$. Let $p\in \Sec_{5}(SV_{(1,1,1)}^{(2,3,3)})$ be a general point lying on the span of general points $p_1,\dots,p_{5}\in SV_{(1,1,1)}^{(2,3,3)}$. By Lemma \ref{lemma_RC} there is a rational curve in $\mathbb{P}^{2}\times\mathbb{P}^3\times\mathbb{P}^{3}$ of multi-degree $(2,3,3)$ through $5$ general points and via the Segre-Veronese embedding we get a rational normal curve $C\subseteq SV_{(1,1,1)}^{(2,3,3)}$ of degree eight through $p_1,\dots,p_{5}$.

Now, $C$ spans a linear space $\Pi\cong\mathbb{P}^{8}$ passing through $p$. Any $4$-dimensional linear subspace of $\Pi$ passing through $p$ that is $5$-secant to $C$ is $5$-secant to $SV_{(1,1,1)}^{(2,3,3)}$ as well. Hence, if this family of $4$-dimensional linear spaces has positive dimension we get that $SV_{(1,1,1)}^{(2,3,3)}$ is $5$-defective. To conclude it is enough to observe that by \cite[Theorem 3.1]{MM13} such family has dimension one.

Now, consider $SV_{(1,1,1)}^{(2,2,2)}$. We may move four general points of $SV_{(1,1,1)}^{(2,2,2)}$ on the diagonal. This a Veronese variety $V^2_{3}$ spanning a linear subspace $\Pi\cong\mathbb{P}^{9}$. Arguing as in the first part of the proof we have that if the family of $3$-dimensional linear subspaces of $\Pi$ through a general point of $\Pi$ and $4$-secant to $V^2_{3}$ form a family of positive dimension then $SV_{(1,1,1)}^{(2,2,2)}$ is $4$-defective. To conclude it is enough to observe that by \cite[Proposition 1.2]{MM13} such family is $2$-dimensional.      
\end{proof}

\begin{Remark}\label{rem1}
The $4$-defectiveness of $SV_{(1,1,1)}^{(2,2,2)}$ was already well known thanks to an explicit equation for $\Sec_4(SV_{(1,1,1)}^{(2,2,2)})$ originally worked out by V. Strassen \cite{Str88} and then generalized by J. M. Landsberg, L. Manivel and G. Ottaviani \cite{CGO14}, \cite{LM08}, \cite{Ott09}, \cite{LO13}. The $5$-defectiveness of $SV_{(1,1,1)}^{(2,3,3)}$ was already known \cite[Proposition 4.10]{AOP09}.
\end{Remark}

\stepcounter{thm}
\subsection{On secant defectiveness of $SV^{(1,1,1)}_{(d_1,d_2,d_3)}$}\label{Cremona}
Let $\mathbb P := \mathbb P^1\times\mathbb P^1\times\mathbb P^{1}$, $H_i$ the pull-back of a hyperplane on the $i$-th factor of $\mathbb P$, and $p_1,p_{2}\in\mathbb P$ general points. Denote by $\mathcal L(a,b,c;2^r)$ the non complete linear system $|aH_1+bH_2+cH_3-
\sum_{i=1}^r2p_i|$ on $\mathbb P$, and let $X\to\mathbb P$ be the blow-up of $\mathbb{P}$ at $p_1,p_2$ with exceptional divisors $E_1, E_2$. Without loss of generality we may take $p_1 = ([0:1],[0:1],[0:1]), p_2 = ([1:0],[1:0],[1:0])$. Consider the rational map
$$
\begin{array}{cccc}
\phi: & \mathbb{P} & \dasharrow & \mathbb{P}\\
 & ([x_0:x_1],[y_0:y_1],[z_0:z_1]) & \mapsto & ([x_1y_0:x_0y_1],[y_0:y_1],[y_0z_1:y_1z_0])
\end{array}
$$
Note that $\phi^2([x_0:x_1],[y_0:y_1],[z_0:z_1]) = \phi([x_1y_0:x_0y_1],[y_0:y_1],[y_0z_1:y_1z_0]) = ([x_0y_1y_0:x_1y_0y_1],[y_0:y:_1],[y_0y_1z_0:y_1y_0z_0]) = ([x_0:x_1],[y_0:y_1],[z_0:z_1])$. So $\phi$ is an involution. 

Then the exceptional locus of $\phi$ is the inverse image via $\phi$ of the indeterminacy locus of $\phi^{-1} = \phi$. Such indeterminacy locus is given by 
$$\{x_1y_0 = x_0y_1 = 0\} = \{[0:1]\times [0:1]\times \mathbb{P}^1\}\cup \{[1:0]\times [1:0]\times \mathbb{P}^1\}$$
$$\{y_0z_1 = y_1z_0 = 0\} = \{\mathbb{P}^1\times [0:1]\times [0:1]\}\cup \{\mathbb{P}^1\times [1:0]\times [1:0]\}$$
Hence the exceptional locus of $\phi$ is given by
$$\{\mathbb{P}^1\times [0:1]\times\mathbb{P}^1\}\cup \{\mathbb{P}^1\times [1:0]\times\mathbb{P}^1\}$$
In particular $\phi$ lifts to a birational, but not biregular, involution $\widetilde{\phi}:X\dasharrow X$, mapping $\{\mathbb{P}^1\times [0:1]\times\mathbb{P}^1\}$ to $E_1$ and $\{\mathbb{P}^1\times [1:0]\times\mathbb{P}^1\}$ to $E_2$, which is an isomorphism in codimension one. The action of $\widetilde{\phi}$ on $\Pic(X)\cong\mathbb{Z}[H_1,H_2,H_3,E_1,E_2]$ is given by the following matrix   
$$
\left(
 \begin{array}{ccccc} 
  1& 0& 0& 0& 0\\
  1& 1& 1& 1& 1\\
  0& 0& 1& 0& 0\\ 
  -1& 0& -1& -1& 0\\
  -1& 0& -1& 0&-1
  \end{array}
\right)
$$
where we keep denoting by $H_1,H_2,H_3$ their pull-backs on $X$. Therefore, $\widetilde{\phi}$ maps the linear system $\mathcal{L}(d_1,d_2,d_3;m_1,m_2)$ to the linear system $\mathcal{L}(d_1,d_1+d_2+d_3-m_1-m_2,d_3;d_1+d_3-m_1,d_1+d_3-m_2)$.

Now, consider a linear system of the form $\mathcal{L}(d_1,d_2,d_3;2^{2r})$ that is with $2r$ double base points. Applying the map $\phi$ centered at two of the double points we get $\mathcal{L}(d_1,d_1+d_2+d_3-4,d_3;d_1+d_3-2,d_1+d_3-2,2^{2r-2})$. Now, applying again the map $\phi$ centered at two of the remaining double points to this new linear system we get $\mathcal{L}(d_1,2d_1+d_2+2d_3-8,d_3;d_1+d_3-2,d_1+d_3-2,d_1+d_3-2,d_1+d_3-2,2^{2r-4})$. Proceeding in this way, after $r$ steps, we get the linear system $\mathcal{L}(d_1,rd_1+d_2+rd_3-4r,d_3;(d_1+d_3-2)^{2r})$. Summing up applying $r$ maps of type $\phi$ we have 
\stepcounter{thm}
\begin{equation}\label{trsys1}
\mathcal{L}(d_1,d_2,d_3;2^{2r})\mapsto \mathcal{L}(d_1,rd_1+d_2+rd_3-4r,d_3;(d_1+d_3-2)^{2r})
\end{equation}
Similarly applying $r$ maps of type $\phi$ to a linear system with an odd number of double base points we get 
\stepcounter{thm}
\begin{equation}\label{trsys2}
\mathcal{L}(d_1,d_2,d_3;2^{2r+1})\mapsto \mathcal{L}(d_1,rd_1+d_2+rd_3-4r,d_3;(d_1+d_3-2)^{2r},2)
\end{equation}

For instance (\ref{trsys1}) yields that $\mathcal{L}(1,d,1;2^{2r})$ goes to $\mathcal{L}(1,d-2r,1)$ and this last linear system has the expected dimension. So by Terracini's lemma \cite{Te11} $SV^{(1,1,1)}_{(1,d,1)}$ is not $2r$-defective for any $r$. Note that since $SV^{(1,1,1)}_{(1,d,1)}\subseteq\mathbb{P}^{4(d+1)-1}$ when $d$ is odd we get that $SV^{(1,1,1)}_{(1,d,1)}$ is not $h$-defective for any $h$ while when $d = 2a$ is even we miss the last secant variety, namely the $(2a+1)$-secant variety, which is indeed defective. To see this note that the linear system $\mathcal{L}(1,2a,1;2^{2a+1})$ is equivalent to $\mathcal{L}(1,0,1;2)$, the $(2a+1)$-secant variety of $SV^{(1,1,1)}_{(1,2a,1)}$ is expected to fill the ambient space $\mathbb{P}^{8a+3}$ but by considering the tangent plane to the quadric surface given by the first and the third copies of $\mathbb{P}^1$ we see that $\mathcal{L}(1,0,1;2)$ has one non trivial section. 

Similarly, $\mathcal{L}(1,d,2;2^{2r})$ goes to $\mathcal{L}(1,d-r,2;1^{2r})$ which has the expected dimension. In this case we get that $SV^{(1,1,1)}_{(1,d,2)}\subseteq\mathbb{P}^{6(d+1)-1}$ is not $h$-defective for any $h\leq \overline{h}$ where $\overline{h}$ is the biggest even number such that $\overline{h}\leq \frac{3}{2}(d+1)$.

Furthermore, (\ref{trsys2}) yields that $\mathcal{L}(1,d,1;2^{2r+1})$ goes to $\mathcal{L}(1,d-2r,1;2)$ which is empty for $2r > d$. So $\mathbb{S}ec_{d+2}(SV^{(1,1,1)}_{(1,d,1)})$ fills the ambient space $\mathbb{P}^{4(d+1)-1}$. However, as we have seen $\mathbb{S}ec_{d+1}(SV^{(1,1,1)}_{(1,d,1)})$ does not fill the ambient space when $d$ is even.

Finally, $\mathcal{L}(1,d,2;2^{2r+1})$ goes to $\mathcal{L}(1,d-r,2;1^{2r},2)$ which is empty for $r > d$. Hence $\mathbb{S}ec_{2d+3}(SV^{(1,1,1)}_{(1,d,2)})$ fills the ambient space $\mathbb{P}^{6(d+1)-1}$.

\begin{Remark}\label{rem2}
We believe that it should be possible to produce rational maps, in the same spirit of what we did in Section \ref{Cremona} for the case of $\mathbb{P}^1\times\mathbb{P}^1\times\mathbb{P}^1$, in order to explain most of the possible new defective cases in the tables in Section \ref{Magma}.

Finally, we would like to stress that the defectiveness of the Segre-Veronese varieties considered in Section \ref{Cremona} was already well known \cite[Theorem 2.1]{LP13}.
\end{Remark}

\bibliographystyle{amsalpha}
\bibliography{Biblio}

\end{document}